\documentclass[a4paper,11pt]{article}
\usepackage[T1]{fontenc}
\usepackage[utf8x]{inputenc} 
\usepackage{lmodern}
\usepackage{amsmath}
\usepackage{amsthm}
\usepackage{amssymb}
\usepackage{authblk}
\usepackage{graphicx}
\usepackage{enumitem}
\usepackage{xcolor}

\usepackage{amsthm}
\usepackage{amssymb}
\usepackage{xparse}
\usepackage{thmtools}
\usepackage{stackrel}
\usepackage{bbold}
\usepackage{relsize}
\usepackage{tikz}

\usetikzlibrary{hobby}

\newcommand{\R}{\mathbb{R}}
\newcommand{\C}{\mathbb{C}}

\newcommand{\N}{\mathbb{N}}

\newcommand{\dd}{\mathrm{d}}

\renewcommand{\S}{\mathbb{S}}

\DeclareMathOperator{\supp}{supp}
\DeclareMathOperator{\diag}{diag}

\DeclareMathOperator{\dist}{dist}

\makeatletter
\newtheoremstyle{indented}
{7pt} 
{7pt} 
{} 
{1.5em} 
{\bfseries} 
{.} 
{.5em} 
{} 

\theoremstyle{definition}

\newtheorem{defn}{Definition}[section]

\theoremstyle{plain}
\newtheorem*{theorem*}{Theorem}

\newtheorem{theorem}{Theorem}

\newtheorem{prop}[defn]{Proposition}

\newtheorem{lem}[defn]{Lemma}

\theoremstyle{definition}
\newtheorem{rem}[defn]{Remark} 

\makeatletter
\renewcommand*\env@matrix[1][*\c@MaxMatrixCols c]{%
  \hskip -\arraycolsep
  \let\@ifnextchar\new@ifnextchar
  \array{#1}}
\makeatother

\usepackage{a4wide}
\title{Toeplitz operators with analytic symbols}
\author{Alix Deleporte\thanks{alix.deleporte@math.uzh.ch}}
\affil{Universit\'e de Strasbourg, CNRS, IRMA UMR 7501, F-67000
  Strasbourg, France\\
  \vspace{1em}
Institut f\"ur Mathematik, Universit\"at Z\"urich, Winterthurerstrasse 190,
CH-8057 Z\"urich, Switzerland}

\usepackage{bbm}

\newcommand{\fixitemthm}{\leavevmode\vspace{-0.5em}} 

\newcommand\blfootnote[1]{%
  \begingroup
  \renewcommand\thefootnote{}\footnote{#1}%
  \addtocounter{footnote}{-1}%
  \endgroup
}

\begin{document}

\maketitle

\blfootnote{This work was supported by grant
  ANR-13-BS01-0007-01}
\begin{abstract}
We develop a new semiclassical calculus in analytic regularity, and we
apply these techniques to the study of Berezin-Toeplitz quantization
in real-analytic regularity.
  
We provide asymptotic formulas for
the Bergman projector and Berezin-Toeplitz operators on a compact
K\"ahler manifold. These objects depend on an integer $N$ and we
study, in the limit $N\to +\infty$, situations in which one can
control them up to an error $O(e^{-cN})$ for some $c>0$.

We develop a calculus of Toeplitz
operators with real-analytic symbols, which applies to K\"ahler manifolds
with real-analytic metrics. In particular, we prove that the Bergman kernel
is controlled up to $O(e^{-cN})$ on any real-analytic K\"ahler manifold as
$N\to +\infty$. We also prove that Toeplitz operators with analytic
symbols can be composed and inverted up to $O(e^{-cN})$. As an
application, we study eigenfunction concentration for Toeplitz
operators if both the manifold and the symbol are real-analytic. In
this case we prove exponential decay in the classically forbidden
region.
\end{abstract}

\section{Introduction}
\label{sec:introduction}

Toeplitz quantization
associates, to a real-valued function $f$ on a K\"ahler
manifold $M$, a family
of \emph{Toeplitz operators}, which are self-adjoint linear operators
$(T_N(f))_{N\geq 1}$ acting on holomorphic sections over $M$. Examples of Toeplitz operators are spin
operators (where $M=\S^2$), which
are indexed by the total spin $\frac N2\in \frac 12 \N$. Toeplitz
operators also include pseudodifferential operators on $\R^n$. 
In this paper we study \emph{exponential estimates} for these operators, that is, approximate
expressions with $O(e^{-cN})$ remainder for some $c>0$.

The family of holomorphic section spaces in Toeplitz quantization is described by a sequence of \emph{Bergman
  projectors} $(S_N)_{N\geq 1}$ 
 (also known as Szeg\H{o} projectors). The operators $S_N$ can be written as
integral operators (the integral kernels are sections of suitable bundles over
$M\times M$), and a first step toward understanding Toeplitz
quantization is the asymptotic study, in the limit $N\to +\infty$, of
the Bergman kernel.

We show (Theorem
\ref{thr:Szeg-gen}) that the Bergman kernel
admits an asymptotic expansion in decreasing powers of $N$, up to an error
$O(e^{-cN})$, as soon as the K\"ahler manifold is compact and real-analytic. To study
the Bergman projector,
as well as compositions of Toeplitz operators (Theorem \ref{thr:Compo}), it is useful to
interpret the $N\to +\infty$ limit as a semiclassical limit (with
semiclassical parameter $\hbar=\frac{1}{N}$). We build new semiclassical tools in
real-analytic regularity (in particular, new analytic symbol classes,
see Definition \ref{def:anal-symb}), which can be of more general use.

This study of the calculus of Toeplitz operators allows us to state
results concerning sequences of eigenfunctions of Toeplitz operators
$(T_N(f))_{N\geq 1}$ for a real-analytic $f$. We prove the following (Theorem
\ref{thr:Gen-exp-decay}): if $(u_N)_{N\geq 1}$ is a sequence of
normalised eigenfunctions with energy near $E\in \R$, that is,
\[
  T_N(f)u_N=\lambda_Nu_N,\qquad\qquad
  \lambda_N\stackrel[N\to +\infty]{}{\to}E,\qquad\qquad
  \|u_N\|_{L^2(M,L^{\otimes N})}=1,
\]
and if $V\subset M$ is
an open set at positive distance from $\{x\in M,f(x)=E\}$,
then \[\|u_N\|_{L^2(V,L^{\otimes N})}\leq Ce^{-cN}\] for some $C>0,
c>0$ independent on $N$. We say that
$(u_N)_{N\geq 1}$ has
an \emph{exponential decay rate} on $V$.

In \cite{deleporte_wkb_2019}, we provide an asymptotic expansion, with error
$O(e^{-cN})$, for the
ground state of a Toeplitz operator $T_N(f)$, for $f$ real-analytic
and Morse.


\subsection{Bergman kernels and Toeplitz operators}
\label{sec:toeplitz-operators}
This article is devoted to the study of exponential estimates
concerning the Bergman kernel and Toeplitz operators on K\"ahler
manifolds with real-analytic data. In this subsection we quickly recall the
framework of Toeplitz operators, introduced in \cite{boutet_de_monvel_spectral_1981,bordemann_toeplitz_1994}. We refer the reader to more
detailed introductions \cite{borthwick_introduction_2000,charles_berezin-toeplitz_2003,deleporte_low-energy_2017}.

\begin{defn}\label{def:toep}\fixitemthm
  \begin{itemize}
    \item A compact K\"ahler manifold $(M,J,\omega)$ is said to be
      quantizable when the
symplectic form $\omega$ has integer cohomology: there exists a
unique Hermitian line bundle $(L,h)$ over $M$ such that the curvature
of $h$ is $-2i\pi
\omega$. This line bundle is called the prequantum line bundle over
$(M,J,\omega)$. The manifold $(M,J,\omega)$ is said to be \emph{real-analytic} when
$\omega$ (or, equivalently, $h$) is real-analytic on the complex
manifold $(M,J)$.
\item Let $(M,J,\omega)$ be a quantizable compact K\"ahler manifold
  with $(L,h)$ its prequantum bundle and let
  $N\in \N$.
  \begin{itemize}
  \item The Hilbert space $L^2(M,L^{\otimes N})$ is the closure of
    $C^{\infty}(M,L^{\otimes N})$, the space of smooth sections of
    the $N$-th tensor power of $L$, for the scalar product
    \[
      \langle u,v\rangle=\int_M\langle u(x),v(x)\rangle_{L^{\otimes
          N}_x}\frac{\omega^{\wedge \dim_{\C}M}}{(\dim_{\C}M)!}.
    \]
    \item The Hardy space $H^0(M,L^{\otimes N})$ is the
space of holomorphic sections of $L^{\otimes N}$. It is a
finite-dimensional, closed subspace of $L^2(M,L^{\otimes N})$.
  \item The Bergman projector $S_N$ is the orthogonal projector from
    the space
  $L^2(M,L^{\otimes N})$ to its subspace $H^0(M,L^{\otimes N})$.

  \item The contravariant Toeplitz operator associated with a
symbol $f\in L^{\infty}(M,\C)$ is
defined as
\begin{align*}
  T_N(f):H^0(M,L^{\otimes N})&\to H^0(M,L^{\otimes N})\\
  u&\mapsto S_N(fu).
\end{align*}
\end{itemize}
\end{itemize}
\end{defn}
In a related
way, one can define \emph{covariant} Toeplitz operators, which are
kernel operators acting on $H^0(M,L^{\otimes N})$ (see
Definition \ref{def:Anal-Toep}).
We are interested  the Bergman projector and both types  of Toeplitz
operators in the \emph{semiclassical limit} $N\to
+\infty$.

A particular motivation for the study of Toeplitz operators is the
quantization, on $M=(\S^2)^{d}$, of polynomials in the coordinates (in the standard immersion of
$\S^2$ into $\R^3$). The operators obtained are spin operators, with
total spin $\frac N2$. Tunnelling effects in spin
systems, in the large spin limit, are widely studied in the physics
literature (see \cite{owerre_macroscopic_2015} for a review). This article also aims at giving a mathematical ground
to this study.

The Bergman kernel plays a role in many aspects of complex geometry
and complex algebraic geometry
\cite{tsuji_dynamical_2010,ross_asymptotics_2016} as well as random
matrices \cite{ameur_fluctuations_2011,klevtsov_random_2014}, expanding the range of potential
applications for Theorem \ref{thr:Szeg-gen}. Beyond the statements of
our main results,  our new microlocal analytic
tools, developed in Section \ref{sec:calc-analyt-symb} may be used
again in many different contexts, including transfer operators and
quantized symplectomorphisms. As a matter of fact, we only deal
with direct summation techniques, and the involved techniques of
resummation or resurgence might be used to further broaden the range
of applications of our tools.

 We will use the
following estimate on the operator $\overline{\partial}$ acting on
$L^2(M,L^{\otimes N})$ and the Bergman projector $S_N$.

\begin{prop}\label{prop:Kohn}
Let $(M,\omega,J)$ be a compact quantizable K\"ahler manifold and
$(S_N)_{N\geq 1}$ be the associated sequence of Bergman projectors. There exists $C>0$ such that, for every
$N\geq 1$ and $u\in L^2(M,L^{\otimes N})$, one has:
\begin{equation}
  \label{eq:Kohn}
  \|\overline{\partial}u\|_{L^2}\geq C\|u-S_Nu\|_{L^2}.
\end{equation}
\end{prop}
This estimate initially follows from the work of Kohn
\cite{kohn_harmonic_1963,kohn_harmonic_1964}; it is
widely used in the asymptotic study of the Bergman kernel, where it is
sometimes named after H\"ormander or Kodaira.

The Bergman projector $S_N$ admits a kernel, in a sense which we make
precise here. The space $H^0(M,L^{\otimes N})$ is finite-dimensional,
so that it
is spanned by a Hilbert basis $s_1,\ldots,s_{d_N}$ of holomorphic
sections of $L^{\otimes N}$. The following section of $L^{\otimes
  N}\boxtimes \overline{L}^{\otimes N}$ is the integral kernel of the
Bergman projector:
\[
  S_N(x,y)=\sum_{i=1}^{d_N}s_i(x)\otimes \overline{s_i(y)}.
\]

Here $\overline{L}$ is the complex conjugate bundle of $L$, and
$\boxtimes$ stands for pointwise direct product: $L^{\otimes N}\boxtimes
\overline{L}^{\otimes N}$ is a bundle over $M\times M$.
More generally, any section of $L^{\otimes N}\boxtimes
\overline{L}^{\otimes N}$ gives rise to an operator on
$L^2(M,L^{\otimes N})$.

\subsection{Statement of the main results}
\label{sec:stat-main-results}

We begin with the definition of what will be the phase of the
Bergman kernel. We use the standard notion of holomorphic extensions of
real-analytic functions and manifolds, under a notation convention
which is recalled in detail in Section \ref{sec:extensions-manifolds}.

\begin{defn}\label{def:natural-section}
  Let $M$ be a real-analytic K\"ahler manifold. Let
  $\overline{M}=(M,\omega,-J)$ be the complex conjugate of $M$:
  holomorphic data on $M$ correspond to anti-holomorphic data on
  $\overline{M}$. The codiagonal $\{(x,x),x\in
  M\}\subset M\times \overline{M}$ is a totally
  real submanifold.

  In particular, there exists a neighbourhood $U$ of the diagonal in
  $M\times \overline{M}$ and a unique holomorphic section $\Psi$ of
  $L\boxtimes \overline{L}$ over $U$ such that $\Psi=1$ on the
  diagonal. Its $N$-th tensor power $\Psi^N$ is the unique holomorphic
  section of $L\boxtimes \overline{L}$ over $U$ such that $\Psi^N=1$
  on the codiagonal.
\end{defn}
It is well-known that the pointwise norm of $\Psi$ decays away from
the codiagonal:
\[
  |\Psi(x,y)|_h=e^{-\frac 12 \dist(x,y)^2+O(\dist(x,y)^3)}.
\]
In the general setting of a K\"ahler manifold with real-analytic data, it
has been conjectured by S. Zelditch that the Bergman
kernel takes the following form: for some $c>0,c'>0$, for all $(x,y)\in M^2$,
\[
  S_N(x,y)=\Psi^N(x,y)\sum_{k=0}^{cN}N^{d-k}a_k(x,y)+O(e^{-c'N}),
\]
where the $a_k$ are, in a neighbourhood of the diagonal in
$M\times M$, holomorphic in the first variable and anti-holomorphic in
the second variable, with
\[
  \|a_k\|_{C^0}\leq CR^kk!.
\]

The well-behaviour of such sequences of functions when the sum $\sum
N^{-k}a_k$ is computed up to the rank $k=cN$ with $c<e/2R$ is well
described in
\cite{sjostrand_singularites_1982} and is the foundation for a theory of
analytic pseudodifferential operators and Fourier Integral
Operators. Here, we rely on more specific function classes.
Without giving a precise definition at this stage let us call 
``analytic symbols'' well-controlled sequences of real-analytic
functions. See Definition \ref{def:anal-symb} about the analytic symbol spaces
$S^{r,R}_m(X)$ and the associated summation. The introduction of these
classes allows us to prove the conjecture.

\begin{theorem}
  \label{thr:Szeg-gen}
  Let $M$ be a quantizable compact real-analytic K\"ahler manifold of
  complex dimension $d$. There
  exists positive constants $r,R,m,c,c',C$, a neighbourhood $U$ of the
  diagonal in $M\times M$, and an analytic symbol
  $a\in S^{r,R}_m(U)$, holomorphic in the first variable,
  anti-holomorphic in the second variable, such that the Bergman
  kernel $S_N$ on $M$ satisfies, for each $x,y\in M\times M$ and
  $N\geq 1$:
  \[
    \left\|S_N(x,y)-\Psi^N(x,y)\sum_{k=0}^{cN}N^{d-k}a_k(x,y)\right\|_{h^{\otimes
        N}}\leq
    C e^{-c'N}.
    \]
  \end{theorem}
  Equivalently, the operator with kernel
  $\Psi^N(x,y)\sum_{k=0}^{cN}N^{d-k}a_k(x,y)$ is exponentially close
  (in the $L^2\to L^2$ operator sense) to the Bergman projector.
  
   Theorem \ref{thr:Szeg-gen} also appears in recent and independent work
  \cite{rouby_analytic_2018}, where the authors use Local Bergman
  kernels as developed in \cite{berman_direct_2008} to study locally
  the Bergman kernel as an analytic Fourier Integral Operator. Here,
  we obtain it as a byproduct of the next theorem about composition
  and inversion of Toeplitz operators.

  In order to study contravariant Toeplitz operators of Definition
  \ref{def:toep}, as well as the Bergman kernel itself, it is useful
  to consider \emph{covariant} Toeplitz operators
  \cite{charles_berezin-toeplitz_2003}. With $\Psi^N$ as above, and
  $f:M\times \overline{M}\to \C$ holomorphic near the diagonal, we
  let
\[
  T_N^{cov}(f)(x,y)=\Psi^N(x,y)\left(\sum_{k=0}^{cN}
    N^{d-k}f_k(x,y)\right),
\]
for some small $c>0$; see Definition \ref{def:Anal-Toep}.
\begin{theorem}
  \label{thr:Compo}
  Let $M$ be a quantizable compact real-analytic K\"ahler manifold.
  Let $f$ and $g$ be analytic symbols on a neighbourhood $U$ of the
  diagonal in $M\times M$,
  which are holomorphic in the first variable and
  anti-holomorphic in the second variable.

  Then there exists $c'>0$ and an analytic
  symbol $f\sharp g$ on the same neighbourhood $U$, holomorphic in the first variable and
  anti-holomorphic in the second variable, and
  such that
  \[T_N^{cov}(f) T_N^{cov}(g)=T_N^{cov}(f\sharp g)+O(e^{-c'N}).\]
  For any $r,R,m$ large enough, the product $\sharp$ is a continuous
  bilinear map from $S^{r,R}_m(U)\times S^{2r,2R}_m(U)$ to
  $S^{2r,2R}_m(U)$ (see Definition \ref{def:anal-symb}); the constant $c'$ depends
  only on $r,R,m$.
  
  If the principal symbol of $f$ does not vanish on $M$ then there is an analytic symbol
  $f^{\sharp -1}$ such that, for some $c'>0$, one has
  \[
    T_N^{cov}(f)T_N^{cov}(f^{\sharp -1})=S_N+O(e^{-c'N}).\]
  Given an analytic symbol $f\in S^{r_0,R_0}_{m_0}(U)$ with non-vanishing
  subprincipal symbol, there exists $C>0$ such that for every $r,R,m$ large enough (depending on
  $f,r_0,R_0,m_0$), one has
  \[
    \|f^{\sharp -1}\|_{S^{r,R}_m(U)}\leq C\|f\|_{S^{r,R}_m(U)}.
    \]
  \end{theorem}
  The stationary phase lemma allows one to prove relatively easily
  that the product $(f,g)\mapsto f\sharp g$ is continuous from
  $S^{r,R}_m\times S^{r,R}_m$ to $S^{Cr,CR}_{Cm}$ for some
  $C>0$. Theorem \ref{thr:Compo} is stronger in that respect, since
  the analytic class for $f\sharp g$ is the same as the one for $g$ if
  $f$ is in a significantly better class. We conjecture that, as it is
  the case for pseudodifferential operators on $\R^d$ \cite{boutet_de_monvel_pseudo-differential_1967}, the
  $\sharp$ product is a Banach algebra product in some analytic space,
  that is, is actually continuous from $S^{r,R}_m\times S^{r,R}_m$ to
  $S^{r,R}_m$. This kind of results is subtler than the general
  techniques of analytic microlocal analysis originating from
  \cite{sjostrand_singularites_1982} allow for, and cannot be reached
  from equivalence of analytic quantizations, for instance. Theorem
  \ref{thr:Compo} relies on the Wick property of Toeplitz covariant
  quantization (Proposition \ref{prop:bound-nb-deriv}): as in the
  Moyal product of pseudodifferential operators, to compute the $k$-th term in
  $f\sharp g$ one differentiates $f$ or $g$ at most $k$ times.

  The fact that $f\sharp g$ belongs to the same analytic class as $g$
  in Theorem \ref{thr:Compo} is a key point in our proof of Theorem \ref{thr:Szeg-gen}.

As an application of composition and inversion properties, one can
study the concentration rate of eigenfunctions, in the general case
(exponential decay in the forbidden region) as well as in the particular
case where the principal symbol has a non-degenerate minimum.

\begin{theorem}
  \label{thr:Gen-exp-decay}
Let $M$ be a quantizable compact real-analytic K\"ahler manifold.
   Let $f$ be a real-analytic, real-valued function on
  $M$ and $E\in \R$. Let $(u_N)_{N\geq 1}$ be a normalized sequence of
  $(\lambda_N)_{N\geq 1}$-eigenstates of $T_N(f)$ with
  $\lambda_N\stackrel[N\to +\infty]{}{\to}E$. Then, for every open set $V$ at positive
  distance from $\{f=E\}$ there exist positive constants $c,C$ such that,
  for every $N\geq 1$, one has
  \[
    \int_V\|u_N(x)\|_h^2\frac{\omega^{\wedge n}}{n!}(dx)\leq Ce^{-cN}.
  \]
\end{theorem}
We say informally that, in the forbidden region $\{f\neq E\}$, the
sequence $(u_N)_{N\geq 1}$ has an exponential decay rate.
\subsection{Exponential estimates in semiclassical analysis}
\label{sec:expon-estim-semicl}

Exact or approximate eigenstates of quantum Hamiltonians are
often searched for in the form of a Wentzel-Kramers-Brillouin (WKB) ansatz:
\[
  e^{\frac{\phi(x)}{\hbar}}(a_0(x)+\hbar a_1(x)+\hbar^2
  a_2(x)+\ldots),
\]
where $\hbar$ is the Planck constant, and is very small at the
observer's scale. In the formula above,
$\Re(\phi)\leq 0$ so that this expression is extremely small outside the
set $\{\Re(\phi)=0\}$ where it concentrates.

From this intuition, an interest developed towards decay rates for
solutions of PDEs with small parameters. The most used setting
in the mathematical treatment of quantum
mechanics is the Weyl calculus of pseudodifferential operators \cite{zworski_semiclassical_2012}. 
Typical decay rates in this setting are of order
$O(\hbar^{\infty})$. Indeed, the composition of two pseudodifferential
operators (or, more generally, Fourier Integral Operators) associated
with smooth symbols can only be expanded in powers of $\hbar$
up to an error $O(\hbar^{\infty})$.

In the particular case of a Schr\"odinger operator $P_{\hbar}=-\hbar^2\Delta+V$ where $V$ is a smooth function, one can obtain an \emph{Agmon estimate}
\cite{helffer_multiple_1984}, which is an
$O(e^{\frac{\phi(x)}{\hbar}})$ pointwise control of eigenfunctions of
$P_{\hbar}$ with eigenvalues
close to $E$. Here, $\phi<0$ on $\{V>E\}$. In this setting one can easily conjugate $P_{\hbar}$ with multiplication
operators of the form $e^{-\frac{\phi}{\hbar}}$, which allows to prove
the control above. This conjugation property is not true for
more general pseudodifferential operators. Moreover, Agmon
estimates yield exponential decay in space variables, and give no information about the concentration rate of the
semiclassical Fourier transform, which is only known to decay at
$O(\hbar^{\infty})$ speed outside zero.

In the setting of pseudodifferential operators on $\R^{d}$ with
\emph{real-analytic} symbols, following analytic microlocal techniques
\cite{sjostrand_singularites_1982}, exponential
decay rates in phase space (that is, exponential decay of the FBI or Bargmann
transform) were obtained in \cite{martinez_estimations_1992,martinez_estimates_1994,martinez_precise_1994,martinez_microlocal_1999}.
Exponential estimates in semiclassical analysis have
important applications in physics \cite{chudnovsky_quantum_1988} where they validate the
WKB ansatz which, in turn, yields precise results on spectral gaps or
dynamics of quantum states (quantum tunnelling). Moreover, on the mathematical level, these
techniques can be used to study non-self-adjoint perturbations \cite{hitrik_non-selfadjoint_2004,hitrik_rational_2008} and
resonances
\cite{helffer_resonances_1986,sjostrand_geometric_1990,melin_bohr-sommerfeld_2001,sjostrand_resonances_2003,faure_prequantum_2006}. 

Since exponential decay in phase space for pseudodifferential
operators is defined by means of the FBI or Bargmann transform, it seems natural
to formulate these questions in terms of Bargmann quantization, which then
generalises to Berezin-Toeplitz quantization on K\"ahler manifolds,
where the semiclassical parameter is the inverse of an integer: $\hbar=N^{-1}$. Yet, for instance, the
validity of the WKB ansatz for a Toeplitz operator, at the bottom of a non-degenerate real-analytic
well, was
only performed when the underlying manifold is $\C$ 
(see \cite{voros_wentzel-kramers-brillouin_1989}), and some results were
recently obtained for non-self-adjoint perturbations of Toeplitz
operators on complex one-dimensional tori
\cite{rouby_bohrsommerfeld_2017}. 

The analysis of Toeplitz operators depends on the knowledge of the
Bergman projector, which encodes the geometrical data of the
manifold on which the quantization takes place. The original microlocal techniques for the
study of this projector
\cite{boutet_de_monvel_sur_1975,zelditch_szego_2000,charles_berezin-toeplitz_2003}
allow for a direct control of the Bergman
kernel up to $O(N^{-\infty})$, from which one can deduce
$O(N^{-\infty})$ estimates for composition and eigenpairs of Toeplitz
operators with smooth symbols \cite{le_floch_theorie_2014,deleporte_low-energy_2016,deleporte_low-energy_2017}. Based on analytic
pseudodifferential techniques, local Bergman
  kernels make it possible to show, under real-analyticity
hypothesis, exponential (that is, $O(e^{-cN})$) decay of the coherent
states in Toeplitz quantization \cite{berman_direct_2008}.

There is a recent increase of activity in the topic of exponential
estimates in Toeplitz quantization: control of the Bergman kernel in
real-analytic or Gevrey regularity
\cite{hezari_off-diagonal_2017,hezari_quantitative_2018,hezari_property_2019,charles_analytic_2019}, but also estimates for the localisation of eigenfunctions of
the form $O(e^{-cN^{\alpha}})$ for $C^{\infty}$ or rougher symbols \cite{zelditch_interface_2016,charles_entanglement_2018,kordyukov_semiclassical_2018,deleporte_fractional_2020}.

\begin{rem}[Gevrey case]
 The methods and symbol classes developed in this paper can
 be easily applied to the Gevrey case. $s$-Gevrey symbol classes are
 defined, and studied, by putting all factorials to the power
 $s>1$. $s$-Gevrey functions have almost holomorphic extensions with
 controlled error near the real locus, so that all results in this
 paper should be valid in the Gevrey case under the two following
 modifications:
 \begin{itemize}
 \item The summation of $s$-Gevrey symbols is performed up to
   $k=cN^{\frac 1s}$.
 \item All $O(e^{-c'N})$ controls are replaced with $O(e^{-c'N^{\frac 1s}})$.
 \end{itemize}
 For instance, we conjecture that the Bergman kernel on a quantizable compact Gevrey
 K\"ahler manifold is known up to $O(e^{-c'N^{\frac 1s}})$. Its kernel
 decays at speed $N^{\dim(M)}e^{-\left(\frac 12 -\varepsilon\right)N\dist(x,y)^2}$ as long
 as $\dist(x,y)\leq cN^{-\frac{s-1}{2s}}$. This would improve recent results \cite{hezari_quantitative_2018}.
\end{rem}

\subsection{Outline}
\label{sec:outline}

In Section \ref{sec:holom-extens} we recall the basic properties of
holomorphic extensions of real-analytic functions. Then, in
Section \ref{sec:calc-analyt-symb}, we define analytic symbol classes
for sequences of functions $(f_k)_{k\geq 0}$ and we give a meaning
to the sum $\sum N^{-k}f_k$ up to exponential precision. These symbol
classes are more precise than the ones appearing in the literature
since 
\cite{sjostrand_singularites_1982}. In Section
\ref{sec:szego-kernel-general} we show Theorems \ref{thr:Szeg-gen} and
\ref{thr:Compo}: the Bergman kernel on a
compact quantizable real-analytic K\"ahler manifold, and the composition of
analytic covariant
Toeplitz operators, are known up to
$O(e^{-cN})$ precision, in terms of analytic symbols, from which we
deduce, in Subsection \ref{sec:expon-decay-low}, general exponential decay
(Theorem \ref{thr:Gen-exp-decay}) in the forbidden region, for
covariant as well as contravariant Toeplitz operators with analytic symbols. 

In Sections \ref{sec:calc-analyt-symb} and those that follow, the fundamental tool is a version in real-analytic
regularity of the stationary phase lemma (Lemma \ref{prop:HSPL}).
The various proofs in the second part have a common
denominator: the general strategy consists in applying the complex
stationary phase lemma and controlling the growth of the derivatives
of the successive terms.


\section{Holomorphic extensions}
\label{sec:holom-extens}

In this section we provide a general formalism for holomorphic
extensions of various real-analytic data, which we use throughout this
paper. The constructions of holomorphic extensions of real-analytic
functions and manifolds is somewhat standard. We refer to
\cite{whitney_quelques_1959} for details on these constructions. In
particular, we study in Subsection \ref{sec:norms-analyt-funct} a
specific class of analytic function spaces, which is a prerequisite to
the Definition \ref{def:anal-symb} of analytic symbol classes.

\subsection{Combinatorial notations and inequalities}
\label{sec:comb-notat}
In this subsection we recall some basic combinatorial notation. Analytic
functions and analytic symbol spaces are defined using sequences which
grow as fast as a factorial (see Definitions \ref{def:Anal-fct-space} and
\ref{def:anal-symb}) so that we will frequently need to bound
expressions involving
binomial or multinomial coefficients.
\begin{defn}
  Let $0\leq i \leq j$ be integers. The associated \emph{binomial} coefficient is
  \[
    \binom{j}{i}=\cfrac{j!}{i!(j-i)!}.
  \]

  Let more generally $(i_k)_{1\leq k \leq n}$ be a family of
  non-negative integers and let $j\geq \sum_{k=1}^ni_k$. The
  associated \emph{multinomial} coefficient is 
  \[
    \binom{j}{i_1,\ldots,i_k}=\cfrac{j!}{(j-\sum_{k=1}^n
      i_k)!\prod_{k=1}^ni_k!}
    \]
  \end{defn}
  \begin{rem}
    An alternative definition of
    multinomial coefficient assumes $j=i_1+\ldots+i_n$, in which case
    one defines $\binom{j}{i_1,\ldots,
      i_n}=\cfrac{j!}{i_1!\ldots i_n!}$. The definition we give
    contains this one, and is more consistent with the notation
    for binomial coefficients.
  \end{rem}

  \begin{defn}\fixitemthm
    \begin{enumerate}
      \item
    A \emph{polyindex} (plural: polyindices) $\mu$ is an ordered family
    $(\mu_1,\ldots, \mu_d)$ of non-negative integers (the set of
    non-negative integers is denoted by $\N_0$). The cardinal $d$
    of the family is called the \emph{dimension} of the polyindex (we will
    only consider the case where $d$ is finite).

  \item The norm $|\mu|$ of the polyindex $\mu=(\mu_1,\ldots,\mu_d)$ is
    defined as $\sum_{i=1}^d\mu_i$.

    \item The partial order $\leq$ on polyindices of same dimension
    is defined as follows: $\nu\leq \mu$ when, for every $1\leq i \leq
    d$, one has $\nu_i\leq \mu_i$.
    
    \item The factorial $\mu!$ is defined as
    $\prod_{i=1}^d\mu_i!$. Together with the partial order, this
    allows to extend the notation for binomial coefficients. If
    $\nu\leq \mu$, then we define the associated binomial coefficient
    as
    \[
      \binom{\mu}{\nu}=\frac{\mu!}{\nu!(\mu-\nu)!}
    \]
  \end{enumerate}
  \end{defn}
We state here a few useful inequalities about binomial coefficients.

\begin{lem}\label{lem:control-poly-binom}
  Let $\nu\leq \mu$ be polyindices. Then \[\binom{\mu}{\nu}\leq \binom{|\mu|}{|\nu|}.\]
\end{lem}

  \begin{lem}\label{lem:binom-multi}
    Let $i \leq j$ and $1\leq k \leq l-1$ be integers. Then
    \[
      \cfrac{(i+k-1)!(j+l-i-k-1)!}{i!k!(j-i)!(l-k)!}\leq
      \cfrac{(j+l-2)!}{j!(l-1)!}.
    \]
    In particular, if $a_1,\ldots, a_n$ are nonnegative integers and
    $b_1,\ldots, b_n$ are positive integers, with $\sum_{i=1}^na_i=j$
    and $\sum_{i=1}^nb_i=l$, then
    \[
      \cfrac{(a_1+b_1-1)!\ldots (a_n+b_n-1)!}{a_1!b_1!\ldots
        a_n!b_n!}\leq \cfrac{(j+l-n)!}{j!(l-n+1)!}.
      \]
    \end{lem}
    \begin{proof}
      For the first part, let $k'=k-1$, then
      \[
        \cfrac{(i+k-1)!(j+l-i-k-1)!}{i!k!(j-i)!(l-k)!}=\frac{1}{k(l-k)}\binom{i+k'}{i}\binom{j+l-2-i-k'}{j-i}.
      \]
      Since $1\leq k \leq l-1$ there holds $\frac{1}{k(l-k)}\leq
      \frac{1}{l-1}$. Moreover, from Lemma
      \ref{lem:control-poly-binom}, one has
      \[
        \binom{i+k'}{i}\binom{j+l-2-i-k'}{j-i}\leq
        \binom{j+l-2}{j}=\cfrac{(j+l-1)!}{j!(l-2)!}.
      \]
      Hence,
      \[
        \cfrac{(i+k-1)!(j+l-i-k-1)!}{i!k!(j-i)!(l-k)!}\leq
        \cfrac{(j+l-2)}{j!(l-1)!}.
      \]

      The second part is deduced from the first part by
      induction. Indeed, we just proved that, denoting
      $a_{n-1}'=a_{n-1}+a_n$ and $b_{n-1}'=b_{n-1}+b_n-1$, one has
      \begin{multline*}
         \cfrac{(a_1+b_1-1)!\ldots (a_n+b_n-1)!}{a_1!b_1!\ldots
        a_n!b_n!}\\ \leq  \cfrac{(a_1+b_1-1)!\ldots (a_{n-2}+b_{n-2}-1)!(a_{n-1}'+b_{n-1}'-1)!}{a_1!b_1!\ldots
        a_{n-2}!b_{n-2}!a_{n-1}'!b_{n-1}'!}.
    \end{multline*}
    Here, the sum of the $a_i$'s has not changed but the sum of the
    $b_i$'s has been reduced by one. By induction,
    \[
      \cfrac{(a_1+b_1-1)!\ldots (a_n+b_n-1)!}{a_1!b_1!\ldots
        a_n!b_n!}\leq \cfrac{(j+l-n)!}{j!(l-n+1)!}.
      \]
    \end{proof}

    \begin{lem}\label{lem:hull}
      Let $\ell\geq 2$ and $n\geq 2$ be integers. The set
      \[
    \left\{(i_1,\ldots,i_n)\in \N_0^n,\sum_{k=1}^ni_k=\ell,\text{ at least
        two of them are $\geq 1$}\right\}.\]
  is contained in the convex hull of all permutations of $(\ell-1,1,0,\ldots,0)$.
\end{lem}
\begin{proof}
  Let us call \emph{support} of a tuple $(i_1,\ldots,i_n)$ the number
  of its elements which are non-zero. We will prove by induction on $2\leq k\leq \min(n,\ell)$
  that the convex hull $S$ of the permutations of $(\ell-1,1,0,\ldots,0)$
  contain all tuples of
  support $k$ such that the sum of all elements is $\ell$.

  For $k=2$, we can indeed recover all elements of the form
  $(\ell-x,x,0,\ldots,0)$ for all $1\leq x\leq \ell-1$ by a convex combination of
  $(\ell-1,1,0,\ldots,0)$ and $(1,\ell-1,0,\ldots,0)$.

  We now proceed to the induction. 
  Suppose that $S$ contains all elements of the form
  $(i_1,\ldots,i_{k-1},0,\ldots,0)$ and their permutations. Then, in
  particular, it contains
  $a_0=(\ell-k+2,1,\ldots,1,0,\ldots, 0)$. For every $1\leq j\leq
  k-2$, $S$ also contains the image of $a_0$ by the
  transposition $(k,k-j)$, which we denote by $a_j$. Moreover, $S$ contains $(\frac{\ell}{k-1},\ldots,\frac{\ell}{k-1},0,\ldots,0)$ and its
permutations.
  From the $(a_j)_{0\leq j\leq k-2}$ and
  $(\frac{\ell}{k-1},\ldots,\frac{\ell}{k-1},0,\ldots,0)$, one can
  form the convex combination
  \begin{multline*}
    \cfrac{\ell-k+1}{(\ell-k+2)(k-2)}\sum_{j=0}^{k-2}a_j+\frac{1}{\ell-k+2}\,\left(0,\frac{\ell}{k-1},\ldots,\frac{\ell}{k-1},0,\ldots,0\right)\\=(\ell-k+1,\underbrace{1,\ldots,1}_{k-1},0,\ldots,0).
  \end{multline*}
  
  In particular, $S$ contains all permutations of
  $(l-k+1,1,\ldots,1,0,\ldots,0)$. Thus, $S$ contains all
  elements of support $k$, since the $k$-uple $(l-k,0,\ldots,0)$ and its
  permutations are the extremal points of the convex
  $\{\sum_{j=1}^ki_j=\ell-k\})$. This concludes the induction.
\end{proof}

\subsection{Extensions of real-analytic functions}
\label{sec:extens-analyt-funct}

The fundamental object that one is allowed to extend in a holomorphic
way is a real-analytic function.

\begin{defn}\label{def:Ext-anal-fct}
  Let $f:U\to E$ be a real-analytic function on an open set $U\in
  \R^n$, which takes values into a real or complex Banach space $E$. A \emph{holomorphic extension} of $f$ is a couple
  $(\widetilde{f}, \widetilde{U})$, where $\widetilde{U}$
  is an open set of $\C^n$ and $\widetilde{f}:\widetilde{U}\mapsto
  E\otimes \C$,
  such that
  \begin{itemize}
  \item $\overline{\partial}\widetilde{f}=0.$
  \item $U\subset \widetilde{U}$,
  \item $\widetilde{f}|_{U}=f$
  \end{itemize}
\end{defn}
Naturally, two holomorphic extensions coincide on the connected
components of their intersections which intersect $U$ since, on a
connected open set of $\C^d$, a holomorphic function which vanishes on a
real set vanishes everywhere.

If $E$ is a real Banach space then $E\otimes \C$ is the
complexification of $E$; if $E$ is complex to begin with then
$E\otimes \C=E$.

The local expression of a real-analytic function as a convergent power
series gives a natural and non-ambiguous way to define a holomorphic
extension.

\subsection{Extensions of manifolds}
\label{sec:extensions-manifolds}

\begin{prop}[\cite{whitney_quelques_1959}]
  Let $M$ be a real-analytic manifold. There is a complex manifold
  $(\widetilde{M},J_e)$ with boundary, such that $M$ is a totally real
  submanifold of $\widetilde{M}$. Then $\widetilde{M}$ is called a
  \emph{holomorphic extension} of $M$.
\end{prop}
In this setting, ``totally real'' means that
\[
  \forall x\in M, T_xM\cap J_e(T_xM)=\{0\}.
\]

The extension of real-analytic manifolds is naturally associated with an
extension of their real-analytic functions.
\begin{prop}
    Let $f$ be a real-analytic function on a real-analytic manifold $M$. Then there exists a
    holomorphic function $\widetilde{f}$ on a holomorphic extension $\widetilde{M}$
    of $M$ such that
    $\widetilde{f}|_{M}=f$.
\end{prop}

In the body of this article we will frequently extend real-analytic
functions on holomorphic manifolds. We introduce a convenient notation
to this end, which is reminiscent of Definition \ref{def:natural-section}.
Locally, a real-analytic function $f$ on a complex manifold of
dimension $d$ can be written
as
\[
  f:z\mapsto \sum_{\nu,\rho \in \N_0^d}c_{\nu,\rho}z^{\nu}\overline{z}^{\rho}.
\]
As the function $f$ is not holomorphic, we specifically
write $f(z,\overline{z})$.
There is then a natural notion of an extension
\[
  \widetilde{f}:(z,w)\mapsto \sum_{\nu,\rho\in \N_0^d}c_{\nu,\rho}z^{\nu}w^{\rho}.
\]
This function is holomorphic on a neighbourhood of $0$ in
$\C^{2d}$. It coincides with $\widetilde{f}$, but the totally real
manifold of interest is not $\{\Im(z)=0\}$ anymore but rather
$\{(z,w),w=\overline{z}\}$.

Let $M$ be a complex manifold; using the convention above let us treat
local charts for $M$ and its holomorphic extension $\widetilde{M}$.
A change of charts in $M$ is a biholomorphism $\phi$ which, in the
convention above, depends only on $z$ as a function on
$\widetilde{M}$. The extended biholomorphism $\widetilde{\phi}$
constructed in the previous subsection can be written as
\[(z,w)\mapsto (\phi(z),\overline{\phi(\overline{w})}).\]

Gluing open sets along the charts $\overline{\phi}$ (defined by $\overline{\phi}(z)=\overline{\phi(\overline{z})}$) yields a manifold
$\overline{M}$, and there is a natural identification $M\ni z\mapsto
\overline{z}\in \overline{M}$, so that $\overline{M}$ is simply $M$
with reversed complex structure.

The expression of $\widetilde{\phi}$ above yields \[\widetilde{M}=M\times \overline{M},\] and $M$ sits in $\widetilde{M}$ as
the totally real submanifold \[\{(z,w)\in M\times \overline{M},
  \overline{z}=w\}.\] This copy of $M$ is said to be the codiagonal of
$M\times \overline{M}$.

Any real-analytic function on $M$ can be extended as a holomorphic
function in a neighbourhood of the codiagonal of
$\widetilde{M}$. If the function was holomorphic (on a small open set
of $M$) to begin with, then its extension depends only on the first
variable (on a small open set of $M\times \overline{M}$).

\subsection{Analytic functional spaces}
\label{sec:norms-analyt-funct}

In this subsection we derive a few tools about the study of holomorphic
functions near a compact totally real set. We first fix a notion of
convenient open sets on which our analysis can take place.

\begin{defn}
  A \emph{domain} of $\R^d$ is an open, relatively compact set $U$
  with piecewise smooth boundary.
\end{defn}

Recall that a holomorphic function $f$ near zero can be written as
\[
  f(z)=\sum_{\nu\in \N_0^d}\cfrac{f_{\nu}}{\nu!}z^{\nu}.
\]
Then, in particular $f_{\nu}=\partial^{\nu}f(0)$.
Since $f$ is holomorphic, the sum above converges for $|z|$ sufficiently small. In other terms, there exists
$r>0$ and $C>0$ such that, for every $\nu\in \N_0^d$, one has
\[
  |f_{\nu}|\leq C\nu!r^{|\nu|}.
\]

\begin{defn}\label{def:Anal-fct-space}
For $j\in \N_0$ and $f$ a function on a domain of $\R^d$ of class $C^j$,
we denote by $\nabla^jf$ the function
$(\partial^{\alpha}f(x))_{|\alpha|=j}$, which maps $U$ to $\R^{\binom{j+d-1}{d-1}}$. For $n\in \N$ and $v\in
\R^n$, we denote $\|v\|_{\ell^1}=\sum_{j=1}^n|v_1|+\ldots+|v_n|$.

  Let $m\in \N_0$ and $r>0$. Let $U$ be a domain in $\R^d$. The space $H(m,r,U)$ is defined as the set of
  real-analytic functions on $U$ such that there exists a
  constant $C$ satisfying, for every
  $j\in \N_0$,
  \[
    \sup_{x\in U}\|\nabla^jf(x)\|_{\ell^{1}}\leq \cfrac{Cr^{j}j!}{(j+1)^m}.
  \]
  The space $H(m,r,U)$ is a Banach space for the norm
  $\|\cdot\|_{H(m,r,U)}$ defined as the smallest constant $C$ such
  that the inequality above is true for every $j$.
\end{defn}
Such functions can be extended to a neighbourhood of $U$ in $\C^d$,
with imaginary part bounded by $r^{-1}$ (and by the distance to the
boundary of $U$). The spaces $H(m,r,U)$ are
compactly embedded in each other for the lexicographic order on
$(r,-m)$: if either $r<r'$ or $r=r',m>m'$, then
\[
  H(m,r,U)\subset H(m',r',U).
\]
Introducing a parameter $m$ will allow us to control polynomial
quantities which appear when one manipulates these holomorphic
function spaces, using Lemmas \ref{prop:lem-easy} and
\ref{prop:lem-hard}. They correspond to a regularity condition at the
boundary of a maximal holomorphic extension: for instance, the
function $x\mapsto x\log(x)$ belongs to $H(1,1,(1/2,3/2))$ but not to
$H(m,1,(1/2,3/2))$ for $m>1$.

It will be useful in the course of this paper to consider various
analytic norms for the same function while maintaining a fixed norm. The definition of the spaces
$H(m,r,U)$ immediately imply the following fact.
\begin{prop}\label{prop:move-m-r}
  Let $m_0\in \N_0$ and $r_0>0$. Let $U$ be a domain in $\R^d$. Let $f\in H(m_0,r_0,U)$. Then, for all
  $m\geq m_0$, for all $r\geq r_02^{m-m_0}$, one has $f\in H(m,r,U)$
  with 
\[
\|f\|_{H(m,r,U)}\leq \|f\|_{H(m_0,r_0,U)}.
\]
\end{prop}

The following lemma will be used several times in what
follows.
\begin{lem}\label{prop:lem-easy}
  Let $d\in \N_0$. There exists $C>0$ such that, for
  any $m\geq \max(d+2,2(d+1))$,
  for any $j\in \N_0$, one has
  \[
    \sum_{i=0}^{j}\cfrac{\min(i+1,j-i+1)^d(j+1)^m}{(i+1)^m(j-i+1)^m}\leq
    2+C\frac{3^m}{4^m}.
    \]
  \end{lem}
  We postpone the proof of this lemma until Section
  \ref{sec:calc-analyt-symb}. More specifically, this is a
  particular case of Lemma \ref{prop:lem-hard}.

Analytic function classes form an algebra for $m$ large enough, and nonvanishing
functions can be inverted.
\begin{prop}\label{prop:alg-holo-func}
  There exists $C>0$ such that the following is true. Let $m\geq 2$. Let $r>0$ and let $U$ be a domain in $\R^n$. Let
  $f,g\in H(m,r,U)$. Then $fg\in H(m,r,U)$, and
  \[\|fg\|_{H(m,r,U)}\leq
    C\|f\|_{H(m,r,U)}\|g\|_{H(m,r,U)}.\] The constant $C$ is universal.

  If $f$ is bounded away from zero on $U$, then $f^{-1}\in H(m,r,U)$, with
  \[
    \|f^{-1}\|_{H(m,r,U)}\leq \cfrac{\|f\|_{H(m,r,U)}}{\inf_U(|f|)^2}.
  \]
\end{prop}
\begin{proof}
  Let $f,g\in H(m,r,U)$ and $j\in \N_0$. Then
  \[
    \sum_{|\alpha|=j}|\partial^{\alpha}(fg)|\leq
    \sum_{|\beta+\gamma|=j}\binom{\beta+\gamma}{\beta}|\partial^{\beta}f|\,|\partial^{\gamma}g|
  \]
  By Lemma \ref{lem:control-poly-binom}, one has, for every $\beta$
  and $\gamma$ such that $|\beta+\gamma|=j$,
  \[
    \binom{\beta+\gamma}{\beta}\leq
    \binom{|\beta+\gamma|}{|\beta|}=\binom{j}{|\beta|}.
    \]
    Hence,
    \[
      \sum_{|\alpha|=j}|\partial^{\alpha}(fg)|\leq
  \sum_{i=0}^{|j|}\binom{j}{i}\|\nabla^if\|_{\ell^{1}}\|\nabla^{|\alpha|-i}g\|_{\ell^{1}},
  \]
  so that, for any $j\geq 0$, one has
  \begin{multline*}
    \|\nabla^j(fg)\|_{\ell^{1}}\\\leq
    \|f\|_{H(m,r,U)}\|g\|_{H(m,r,U)}\cfrac{r^jj!}{(j+1)^m}\sum_{i=0}^j\binom{j}{i}^{-1}\binom{j}{i}\cfrac{(j+1)^m}{(i+1)^m(j-i+1)^m}.
  \end{multline*}
  Hence,
  \[
    \|\nabla^j(fg)\|_{\ell^{1}}\leq
    \|f\|_{H(m,r,U)}\|g\|_{H(m,r,U)}\cfrac{r^jj!}{(j+1)^m}\sum_{i=0}^j\cfrac{(j+1)^m}{(i+1)^m(j-i+1)^m}.
  \]
  Let us use Lemma \ref{prop:lem-easy} with $d=0$. If $m\geq 2$, this quantity is
  bounded independently of $j$ and $m$, so that
  \[
    \|\nabla^j(fg)\|_{\ell^{1}}\leq
    C\|f\|_{H(m,r,U)}\|g\|_{H(m,r,U)}\cfrac{r^jj!}{(j+1)^m}.
  \]
  This concludes the first part of the proof.

  Let now $f\in H(m,r,U)$ be bounded away from zero on $U$. We introduce the
  modified product $f\cdot g = \frac{fg}{C}$, for which $H(m,r,U)$
  is a Banach algebra. 

  First, $|f|^2$ is
  real-valued and strictly positive; moreover \[|f|^2=f\overline{f}\in
  H(m,r,U)\] and, by the property above,\[
    \||f|^2\|_{H(m,r,U)}\leq C\|f\|_{H(m,r,U)}^2.
  \]
  Let $g=\frac{|f|^2}{2\||f|^2\|_{H(m,r,U)}}$. Then
    \[
      \|1-g\|_{H(m,r,U)}\leq 1-\frac{\inf_U(|f|^2)}{2\||f|^2\|_{H(m,r,U)}}<1.
      \]
      In particular, $g=1-(1-g)$ so that, letting $h$ be such that $g\cdot
      h=1$, one has
      \[h=\sum_{k=0}^{+\infty}(1-g)^{\cdot k},
      \]
      where the power series converge because the $\cdot$ product
      induces a Banach algebra structure on $H(m,r,U)$.
      
        Hence, one can control
        \[
          \|h\|_{H(m,r,U)}\leq \cfrac{2\||f|^2\|_{H(m,r,U)}}{\inf_U(|f|^2)}.
        \]
        Now $|f|^{-2}=\cfrac{h}{2C\||f|^2\|_{H(m,r,U)}}$ so
        that
        \[
          \||f|^{-2}\|_{H(m,r,U)}\leq
          \cfrac{1}{C\inf_U(|f|^2)}.
        \]
        We now turn to $f^{-1}=\overline{f}|f|^{-2}$, which is
        controlled as follows:
        \[
          \|f^{-1}\|_{H(m,r,U)}\leq
          \cfrac{\|f\|_{H(m,r,U)}}{\inf_U(|f|^2)}.
        \]
        This concludes the proof.
      \end{proof}

The spaces $H(r,m,U)$ contain all holomorphic functions.
\begin{prop}\label{prop:Cauchy}
  Let $d\in \N$. For every $T>0$ we let $P(0,T)$ be the polydisk of center $0$ and of
  radius $T$ in $\C^d$.
  
  Let $f$ be a holomorphic, bounded function on $P(0,2T)$,
  continuous up to the boundary. 
  Then
  \[
    \|f\|_{H(-d,dT^{-1},P(0,T))}\leq C\sup_{P(0,2T)}|f|.
    \]
  \end{prop}
  \begin{proof}
    The proof relies on the Cauchy formula. For all $z\in P(0,T)$ and
    $\nu\in \N_0^d$, there holds
    \[
      \partial^{\nu}f(z)=C\int_{|\xi_1|=\ldots=|\xi_d|=2T}\cfrac{\nu!f(\xi)}{(\xi_1-z_1)^{\nu_1}(\xi_2-z_2)^{\nu_2}\ldots(\xi_d-z_d)^{\nu_d}}\dd
      \xi.
    \]
    As $z\in P(0,r)$ and $|\xi_1|=\ldots=|\xi_d|=2T$, for every $1\leq
    i \leq d$ there holds $|\xi_i-z_i|\geq T$, so that
    \[
    \sup_{P(0,T)}|\partial^{\nu}(f)|\leq
    CT^{-|\nu|}\nu!\sup_{P(0,2T)}|f|.
  \]
  In particular, since $\nu!\leq |\nu|!d^{|\nu|}$, by summing over
  $\nu$'s with same norm we obtain
  \[
    \sup_{x\in P(0,T)}\|\nabla^jf(x)\|_{\ell^1}\leq
    C(j+1)^d(dT^{-1})^{j}j!,
  \]
  hence the claim.
  \end{proof}


\section{Calculus of analytic symbols}
\label{sec:calc-analyt-symb}

In this section we define and study (formal) \emph{analytic symbols}, which we
will show to be well suited to the study of stationary phases with
complex, real-analytic phases.

\subsection{Analytic symbols}
\label{sec:analytic-symbols}

We begin with an explicit definition of $C^j$-seminorms on compact manifolds.
\begin{defn}
  Let $X$ be a compact manifold (with smooth boundary). We fix a finite set
  $(\rho_V)_{V\in \mathcal{V}}$ of local charts on open sets $V$ which
  cover $X$.

  Let $j\geq 0$. The $C^j$ seminorm of a function $f:X\mapsto \C$
  which is continuously differentiable $j$ times is defined as
  \[
    \|f\|_{C^j(X)}=\max_{V\in \mathcal{V}}\sup_{x\in
      V}\sum_{|\mu|=j}|\partial^{\mu}(f\circ \rho_V)(x)|.
  \]
\end{defn}

This definition is adapted to the multiplication of two functions. The
Leibniz formula yields directly:
\begin{prop}
  \label{prop:prod-Cj}
  Let $X$ be a compact manifold (with smooth boundary) with fixed local charts, and $f,g\in C^j(X,\R)$.

  Then $fg\in C^j(X,\R)$ with
  \[
    \|fg\|_{C^j(X)}\leq \sum_{i=0}^j\binom{j}{i}\|f\|_{C^i(X)}\|g\|_{C^{j-i}(X)}.
    \]
\end{prop}

Using the convention above, let us generalise Definition
\ref{def:Anal-fct-space}, in order to define analytic symbols.

\begin{defn}\label{def:anal-symb}
  Let $X$ be a compact manifold (with boundary), with a fixed set of
  covering local charts.

  Let $r, R,m$ be positive real numbers. The space of analytic symbols
  $S^{r,R}_m(X)$ consists of sequences $(a_k)_{k\geq 0}$ of real-analytic
  functions on $X$, such that there exists $C\geq 0$ such that, for
  every $j\geq 0, k\geq 0$, one has
  \[
    \|a_k\|_{C^j(X)}\leq C\cfrac{r^jR^k(j+k)!}{(j+k+1)^m}.
  \]

  The norm of an element $a\in S^{r,R}_m(X)$ is defined as the
  smallest $C$ as above; then $S^{r,R}_m(X)$ is a Banach space.
\end{defn}

We are interested in symbols which have an expansion in
increasing powers of the semiclassical parameter. We will use the term
``symbols'' while, in the usual semiclassical vocabulary, we are dealing
with formal symbols to which we associate classical symbols by a
summation process in Proposition \ref{prop:sum-anal-symb}.

As for the analytic function classes $H(m,r,U)$ of Definition \ref{def:Anal-fct-space}, the spaces $S^{r,R}_m(X)$ are included in each other for
a lexicographic order, and the constants of injection are controlled
as follows:
\begin{prop}\label{prop:move-m-r-R}
  Let $X$ be a compact manifold (with boundary) with a fixed finite
  set of covering charts. Let $r_0,R_0,m_0$ positive.
  Let $f\in S^{r_0,R_0}_{m_0}(X)$. For every $m\geq m_0$, for every
  $r\geq r_02^{m-m_0}$ and $R\geq R_02^{m-m_0}$, one has $f\in
  S^{r,R}_m$ with
  \[
    \|f\|_{S^{r,R}_m(X)}\leq \|f\|_{S^{r_0,R_0}_{m_0}(X)}.
    \]
  \end{prop}

  The notion of sum of a formal series in $N^{-1}$ is well-defined up to $O(N^{-\infty})$, by a
process known as \emph{Borel summation}. In a similar but more
explicit way, formal series corresponding to analytic symbols can be
summed up to an exponentially small error.
\begin{defn}\label{def:sum-anal-symb}
  Let $X$ be a compact Riemannian manifold (with boundary) and let $f\in
  S^{r,R}_m(X)$. Let $c_R=\frac{e}{3R}$. The \emph{summation} of $f$ is
  defined as
  \[
    X\times \N \ni(x,N)\mapsto f(N)(x)=\sum_{k=0}^{c_RN}N^{-k}f_k(x).
    \]
  \end{defn}
  \begin{prop}\label{prop:sum-anal-symb}
    Let $X$ be a compact Riemannian manifold with boundary and let
    $f\in S^{r,R}_m(X)$. Let $c_R=\frac{e}{3R}$. Then
    \begin{enumerate}
    \item The function $f(N)$ is bounded on $X$ uniformly for $N\in \N$.
    \item For every $0<c_1<c_R$, there exists $c_2>0$ such that
      \[
        \sup_{x\in
          X}\left|\sum_{k=c_1N}^{c_RN}N^{-k}f_k(x)\right|=O(e^{-c_2N}).
        \]
    \end{enumerate}
  \end{prop}
  \begin{proof}
    \fixitemthm
    \begin{enumerate}
    \item Since
      \[\sup_{x\in X}|f_k(x)|\leq \|f\|_{S^{r,R}_m(X)}R^kk!,
      \]
      it remains to control
      \[
        \sum_{k=0}^{c_RN}N^{-k}R^kk!.
      \]
      In this series, the first term is $1$, and the ratio between two
      consecutive terms is
      \[
        \cfrac{N^{-k}R^{k}k!}{N^{-k+1}R^{k-1}(k-1)!}=\cfrac{Rk}{N}\leq
        Rc_R=\frac{e}{3}<1.
      \]
      Hence,
      \[\sup_{x\in X}|f(x,N)|\leq
      \|f\|_{S^{r,R}_m(X)}\sum_{k=0}^{c_RN}(e/3)^k\leq
      \|f\|_{S^{r,R}_m(X)}\frac{3}{3-e}.
      \]
    \item The claim reduces to a control on
      \[
        \sum_{k=c_1N}^{c_RN}N^{-k}R^kk!.
      \]
      In this series, on which each term is smaller than $(e/3)^k$,
      the first term is controlled by
      \[
        (e/3)^{c_1N}=\exp(c_1\log(e/3)N).
      \]
      Hence the claim, with $c_2=c_1\log(e/3)$.
      \end{enumerate}
    \end{proof}
    From the second point of Proposition \ref{prop:sum-anal-symb}, we
    see that the constant $c_R=\frac{e}{3R}$ is quite arbitrary (using the
    Stirling formula to control factorials, one could in fact consider any
    constant smaller than $\frac{e}{R}$). We use it in Definition
    \ref{def:sum-anal-symb} to avoid dealing with equivalence classes
    of sequences whose difference is $O(e^{-c'N})$ for some $c'$, as in \cite{sjostrand_singularites_1982}.

Before studying further the space $S^{r,R}_m(X)$, let us prove a
generalisation of
Lemma \ref{prop:lem-easy}.

\begin{lem}
  \label{prop:lem-hard}
  Let $d\in \N$ and $n\geq 2$. There exists $C(n,d)>0$ such that, for any $m\geq \max(d+2,2(d+n-1))$,
  for any $\ell\in \N_0$, one has
  \[
    \sum_{\substack{0\leq i_1\leq i_2\leq \cdots \leq
        i_n\\i_1+\ldots+i_n=\ell}}\cfrac{(i_{n-1}+1)^d(\ell+1)^m}{(i_1+1)^m\ldots(i_n+1)^m}\leq 1+C\frac{3^m}{4^m}.
  \]
\end{lem}
This is indeed, up to a factor $2$, a generalisation of Lemma \ref{prop:lem-easy} which
corresponds to the case $n=2$.
\begin{proof}The case $\ell=1$ is trivial, so we assume
  $\ell\geq 2$.
  The only term in the sum such that $i_{n-1}=0$ is equal to $1$; let us control the sum
  restricted on $\{i_{n-1}\geq 1\}$. Let us first show
  that, if $i_{n-1}\geq 1$, then
  \begin{equation}
\label{eq:bd-interior-terms}
    \cfrac{(i_{n-1}+1)^d(\ell+1)^m}{(i_1+1)^m\ldots(i_n+1)^m}\leq
    (\ell+1)^d\cfrac{3^m}{4^m}.
  \end{equation}
  One has directly
  $
    (i_{n-1}+1)^d\leq (\ell+1)^d.
  $
  
  We are left with
  \[
    \cfrac{(\ell+1)^m}{(i_1+1)^m\ldots(i_n+1)^m},
    \]
    which is a symmetric
  expression of $(i_1,\ldots,i_n)$, log-convex as soon as
  $m\geq 0$, and which we wish to bound on the symmetrised set
  \[
    \left\{(i_1,\ldots,i_n)\in \N_0^n,\sum_{k=1}^ni_k=\ell,\text{ at least
        two of them are $\geq 1$}\right\}.\]
  By Lemma \ref{lem:hull}, it is sufficient to control the quantity
  above at the permutations of $(\ell-1,1,0,\ldots,0)$. At each
  of those points, since $\ell\geq 2$, one has
  \[
\cfrac{(\ell+1)^m}{(i_1+1)^m\ldots(i_n+1)^m}=    \left(\cfrac{\ell+1}{2\ell}\right)^m\leq \cfrac{3^m}{4^m}.
  \]
  
  We are now in position to prove the claim.
  Let us first restrict our attention to $\{i_1\geq \frac{\ell+1}{3(n-1)}\}$. There are less
  than $(\ell+1)^{n-1}$ such terms (since there are less than $(\ell+1)^{n-1}$ terms
  in total), and each of these terms is smaller than
  \[
    \cfrac{(\ell+1)^d(\ell+1)^m}{\left(\frac{\ell+1}{3(n-1)}\right)^{mn}}=\cfrac{(\ell+1)^d\left(3(n-1)\right)^{mn}}{(\ell+1)^{m(n-1)}}.\]
  Hence, this sum is controlled by
  \[
    \cfrac{(\ell+1)^{n+d-1}\left(3(n-1)\right)^{mn}}{(\ell+1)^{m(n-1)}}
  \]

  We now consider the sum on $\{i_1\leq
  \frac{\ell+1}{3n-1}\leq i_2\}$. There are again
  less than $(\ell+1)^{n-1}$ such terms, each of them smaller than
  \[
    \cfrac{(\ell+1)^d(\ell+1)^m}{\left(\frac{\ell+1}{3(n-1)}\right)^{m(n-1)}}=\cfrac{(\ell+1)^d\left(3(n-1)\right)^{m(n-1)}}{(\ell+1)^{m(n-2)}}.\]

  Thus, this sum
  is smaller than
  \[
    \cfrac{(\ell+1)^{n+d-1}\left(3(n-1)\right)^{m(n-1)}}{(\ell+1)^{m(n-2)}}.
  \]

  Similarly, we are able to control the sum on
  $\{i_{k}\leq \frac{\ell+1}{3(n-1)}\leq i_{k+1}\}$, for $k\leq n-2$, by
  \[
    \cfrac{(\ell+1)^{n+d-1}\left(3(n-1)\right)^{m(n-k)}}{(\ell+1)^{m(n-k-1)}}.
  \]
  If $m\geq 2(d+n-1)$, then $(\ell+1)^{n+d-1+m}\leq
  (\ell+1)^{3m/2}$, so that, for any $k\leq n-2$, if $\ell+1\geq 3n$, one has
  \begin{align*}
    \cfrac{(\ell+1)^{n+d-1}\left(3(n-1)\right)^{m(n-k)}}{(\ell+1)^{m(n-k-1)}}&\leq
    (\ell+1)^{\frac{3m}{2}}\left(\cfrac{3(n-1)}{\ell+1}\right)^{m(n-k)}\\
    &\leq (\ell+1)^{3m/2}\left(\frac{3(n-1)}{\ell+1}\right)^{2m}\\ &=\left(\cfrac{9(n-1)^2}{\sqrt{\ell+1}}\right)^m.
  \end{align*}
  Thus, for $\ell$ large enough (depending on $n$), this
  quantity is smaller than
  $\frac{3^m}{4^m}$; for $\ell$ small we have a number of terms bounded by a
  function of $n$, each term being smaller than
  $C(n,d)\frac{3^m}{4^m}$ by \eqref{eq:bd-interior-terms}.

  It remains to control the sum restricted on $\{1\leq i_{n-1}\leq
  \frac{\ell+1}{3(n-1)}\}$. In this case, $i_{n}+1\geq \frac{2(\ell+1)}{3}$, so that the
  sum is smaller than
  \begin{multline*}
    \frac{3^m}{2^m}\sum_{\substack{0\leq i_1\leq \cdots \leq
      i_{n-1}\leq \frac{\ell+1}{3(n-1)}\\i_{n-1}\geq 1}}\cfrac{(i_{n-1}+1)^d}{(i_1+1)^m(i_2+1)^m\ldots
      (i_{n-1}+1)^m}\\ \leq \cfrac{3^m}{2^m}(\zeta(m))^{n-2}(\zeta(m-d)-1).
  \end{multline*}
  The Riemann zeta function is decreasing, and if $m\geq d+2$, then \[\zeta(m-d)\leq 1+3\cdot2^{-(m-d)},\] so that
    the expression above is controlled by
    $C(n,d)\frac{3^m}{4^m}$. This concludes the proof.
\end{proof}

Analytic symbols behave well with respect to the Cauchy product, which
corresponds to the product of their summations.
\begin{prop}\label{prop:anal-symb-class}
  There exists $C_0\in \R$ and a function $C:\R^2\mapsto \R$ such that
  the following is true.

  Let $X$ be a compact Riemannian manifold (with boundary) and with
  a fixed finite set of covering charts. Let $r,R\geq
  0$ and $m\geq 4$. For $a,b\in S^{r,R}_m(X)$, let us define
  the Cauchy product of $a$ and $b$ as
  \[
    (a*b)_k=\sum_{i=0}^ka_ib_{k-i}.
    \]
    \begin{enumerate}
      \item
    The space $S^{r,R}_m(X)$ is an algebra for this Cauchy product, that
    is,
    \[
      \|a* b\|_{S^{r,R}_m}\leq
      C_0\|a\|_{S^{r,R}_m}\|b\|_{S^{r,R}_m},
    \]
      Moreover, there exists $c>0$ depending only on $R$ such that
    as $N$ tends to infinity, one has
    \[
      (a*b)(N)=a(N)b(N)+O(e^{-cN}).
    \]

    \item Let $r_0,R_0,m_0$ positive and $a\in S^{r_0,R_0}_{m_0}(X)$ with $a_0$
    nonvanishing. Then, for every $m$ large enough depending on $a$,
    for every $r\geq r_02^{m-m_0}$ and $R\geq R_02^{m-m_0}$, the
    symbol $a$ is invertible (for the Cauchy product) in
    $S^{r,R}_m(X)$, and its inverse $a^{\star -1}$ satisfies:
    \[
      \|a^{* -1}\|_{S^{r,R}_m(X)}\leq 2\min(|a_0|)^{-4}\|a\|^3_{S^{r_0,R_0}_{m_0}}.
    \]
  \end{enumerate}
\end{prop}
    \begin{proof}\fixitemthm
      \begin{enumerate}
        \item
      From Proposition \ref{prop:prod-Cj}, one has, for every $0\leq i
      \leq k $ and $j\geq 0$,
      \[
        \|a_ib_{k-i}\|_{C^j}\leq
        \sum_{\ell=0}^j\binom{j}{\ell}\|a_i\|_{C^{\ell}}\|b_{k-i}\|_{C^{j-\ell}}.\]
        In particular,
        \begin{multline*}
          \|(a*b)_k\|_{C^j}\leq
          \|a\|_{S^{r,R}_m}\|b\|_{S^{r,R}_m}\cfrac{r^jR^k(j+k)!}{(j+k+1)^m}\\
            \times
            \sum_{i=0}^k\sum_{\ell=0}^j\binom{j+k}{i+\ell}^{-1}\binom{j}{\ell}\cfrac{(j+k+1)^m}{(i+\ell+1)^m(j+k-i-\ell+1)^m}.
          \end{multline*}
        Since,
        \[
          \binom{j}{\ell}\leq \binom{j+i}{\ell+i}\leq
          \binom{j+k}{\ell+i},
          \]
        one has
          \begin{align*}
          \|(a*b)_k\|_{C^j}\leq
                             \|a\|_{S^{r,R}_m}&\|b\|_{S^{r,R}_m}\cfrac{r^jR^k(j+k)!}{(j+k+1)^m}\\
            \times \sum_{i=0}^k\sum_{\ell=0}^j&\cfrac{(j+k+1)^m}{(i+\ell+1)^m(j+k-i-\ell+1)^m}\\
            \leq
            \|a\|_{S^{r,R}_m}&\|b\|_{S^{r,R}_m}\cfrac{r^jR^k(j+k)!}{(j+k+1)^m}\\
            \times \sum_{i'=0}^{k+j}&\cfrac{\min(i'+1,j+k-i'+1)(j+k+1)^m}{(i'+1)^m(j+k-i'+1)^m},
          \end{align*}
          where $i'=i+\ell$.
        We are
        reduced to Lemma \ref{prop:lem-easy} with $d=1$. If $m\geq 4$, this sum is smaller than a universal constant $C$
        independently of $j,k$, so that
        \[
          \|a*b\|_{S^{r,R}_m}\leq C
          \|a\|_{S^{r,R}_m}\|b\|_{S^{r,R}_m}.
        \]

        Let us control the product of the associated analytic
        series. By Proposition \ref{prop:sum-anal-symb}, for some
        $c>0$ depending only on $R$, one has
        \[
          a(N)=\sum_{k=0}^{\frac{eN}{12R}}N^{-k}a_k+O(e^{-cN}),
        \]
        and similar controls for $b(N)$ and $(a* b)(N)$.

        The first $\cfrac{eN}{12R}$ terms of the expansion in
        decreasing powers of $(a*b)(N)$ and $a(N)b(N)$ then coincide by
        definition of the Cauchy product. It remains to control
        \[
          \sum_{\frac{eN}{12R}\leq i+j\leq
            \frac{eN}{6R}}N^{-(i+j)}a_ib_j.
        \]
        From
        \[
          \sup(|a_ib_j|)\leq CR^{i+j}i!j!\leq C(2R)^{i+j}(i+j)!,
        \]
         one has, as in Proposition \ref{prop:sum-anal-symb},
        \[
          \left|\sum_{\frac{eN}{12R}\leq i+j\leq
            \frac{eN}{6R}}N^{-(i+j)}a_ib_j\right|\leq \sum_{\frac{eN}{12R}\leq i+j\leq
            \frac{eN}{6R}}N^{-(i+j)}(2R)^{i+j}(i+j)!\leq e^{-cN},
        \]
        hence the claim.

        \item
        The unit element of the Cauchy product is $(1,0,0,\ldots)$, which belongs to $S^{r,R}_m(X)$. Let
        $a\in S^{r_0,R_0}_{m_0}(X)$ be such that $a_0$ does not vanish on $X$, and let
        us try to find $b$ such that $(a* b)_0=1$ and $(a*
        b)_k=0$ whenever $k\neq 0$.

        The first condition yields $b_0=a_0^{-1}$, which is a function
        with real-analytic regularity and same radius as $a_0$, by
        Proposition \ref{prop:alg-holo-func}, so that
        \[
          \|b_0\|_{C^j}\leq C_0 \cfrac{r_0^jj!}{(j+1)^{m_0}}.
        \]
        In particular, by Lemma \ref{prop:move-m-r}, for all $m\geq
        m_0,r\geq r_02^{m-m_0}$, one has
        \[
          \|b_0\|_{C^j}\leq C_0 \cfrac{r^jj!}{(j+1)^{m}}.
          \]
        
        The coefficients $b_k$ are then determined by induction:
        \[
          b_k=a_0^{-1}\sum_{i=1}^{k}a_ib_{k-i}=b_0\sum_{i=1}^ka_ib_{k-i}.
        \]
        Let us control $\|b\|_{S^{r,R}_m(X)}$ by
        $\|a\|_{S^{r,R}_m(X)}$ by induction, for some $r,R,m$ which
        will be chosen later.
        
        We now proceed by induction on $k$. Suppose that, for all
        $\ell\leq k-1$ and $j\geq 0$,
        one has
        \[
          \|b_{\ell}\|_{C^j}\leq C_b
          \cfrac{r^jR^{\ell}(j+\ell)!}{(j+\ell+1)^m},
        \]
        We wish to prove the same control for $\ell=k$.
        The constant $C_b$ will be chosen later.

        By induction hypothesis,
        \begin{multline*}
          \|b_k\|_{C^j}\leq
          C_0C_b\|a\|_{S^{r,R}_m}\sum_{j_1=0}^j\sum_{i=1}^k\sum_{j_2=0}^{j-j_1}
                         \binom{j}{j_1,j_2}\cfrac{r^{j_1}j_1!}{(j_1+1)^m}\\
                         \times\cfrac{r^{j_2}R^{i}(j_2+i)!r^{j-j_1-j_2}R^{k-i}(j-j_1-j_2+k-i)!}{(i+j_2+1)^m(j-j_1-j_2+k-i+1)^m}\\
          \leq
            C_bC_0\|a\|_{S^{r,R}_m}\cfrac{r^jR^k(j+k)!}{(j+k+1)^m}\hspace{-0.2em}\sum_{j_1=0}^j\sum_{i=1}^k\sum_{j_2=0}^{j-j_1}\binom{j}{j_1,j_2}\binom{j+k}{j_1,j_2+i}^{-1}\\\times\cfrac{(j+k+1)^m}{(j_1+1)^m(j_2+i+1)^m(j-j_1-j_2+k-i+1)^m}.
          \end{multline*}
          Let us prove that, for every $i,j,j_1,j_2,k$ in the range
          above, one has
          \[
            \binom{j+k}{j_1,j_2+i}\geq \binom{j}{j_1,j_2}.
          \]
          There holds
          \[
            \binom{j+1}{j_1,j_2+1}=\binom{j}{j_1,j_2}\cfrac{j+1}{j-j_1-j_2}\geq
            \binom{j}{j_1,j_2},
          \]
          so that
          \[
            \binom{j+k}{j_1,j_2+i}\geq \binom{j+i}{j_1,j_2+i}\geq
            \binom{j}{j_1,j_2}.\]
          Hence,
          \begin{align*}
            \|b_k\|_{C^j}
            \leq
              C_bC_0&\|a\|_{S^{r,R}_m}\cfrac{r^jR^k(j+k)!}{(j+k+1)^m}\\
            \times \sum_{j_1=0}^j\sum_{i=1}^k\sum_{j_2=0}^{j-j_1}&\cfrac{(j+k+1)^m}{(j_1+1)^m(j_2+i+1)^m(j-j_1-j_2+k-i+1)^m}\\
\leq C_bC_0\|a\|_{S^{r,R}_m}&\cfrac{r^jR^k(j+k)!}{(j+k+1)^m}\\ &\times \sum_{\substack{j_1+i_1+i_2=j+k\\i_1\geq 1}}\cfrac{\min(i_1+1,i_2+1)(j+k+1)^m}{(j_1+1)^m(i_1+1)^m(i_2+1)^m}
            .
          \end{align*}
          From Lemma \ref{prop:lem-hard} with $n=3$ and $d=1$, the sum
        \[
         \sum_{\substack{j_1+i_1+i_2=j+k\\i_1\geq 1}}\cfrac{\min(i_1+1,i_2+1)(j+k+1)^m}{(j_1+1)^m(i_1+1)^m(i_2+1)^m}
        \]
        is bounded independently of $j$ and $k$ for $m\geq 6$. However
        this control is not enough since it yields a constant in front of
        $\cfrac{r^jR^k(j+k)!}{(j+k+1)^m}$  which is a priori
        $CC_0C_b\|a\|_{S^{r,R}_m}\geq C_b$.

        However, the only term in this expansion which contributes
        as $1$ is $j_1=0,i_1=k+j,i_2=0$, which corresponds to $j_1=0,i=k,j_2=j$.
        One can control this term independently of $C_b$ since \[
          |a_0^{-1}|\|a_k\|_{C^j}|b_0|\leq
        C_0^2\|a\|_{S^{r_0,R_0}_{m_0}}\cfrac{r^jR^k(j+k)!}{(j+k+1)^m}.
      \]
      The sum over all other terms is smaller than $CC_bC_0\|a\|_{S^{r,R}_m}(3/4)^{m}$ for some
      $C$, by Lemma \ref{prop:lem-hard}.

      We can conclude: if $m$ is large with respect to
      $\|a\|_{S^{r,R}_m}$ (which can be done using Proposition
      \ref{prop:move-m-r-R} by setting $r\geq r_02^{m-m_0}$ and
      $R\geq R_02^{m-m_0}$) and if $C_b\geq
      2C_0^2\|a\|_{S^{r_0,R_0}_{m_0}}$, where we recall from Proposition \ref{prop:alg-holo-func} that \[C_0^2=\min(|a_0|)^{-4}\|a\|_{S^{r_0,R_0}_{m_0}}^2,\] one has, by induction,
      \[
        \|b_k\|_{C^j}\leq C_b\cfrac{r^jR^k(j+k)!}{(j+k+1)^m}.
      \]
      This concludes the proof.
    \end{enumerate}
  \end{proof}

      \begin{rem}
        The method of proof for Proposition
        \ref{prop:anal-symb-class} will be used again in Section
        \ref{sec:szego-kernel-general}. This method consists in an induction,
        in which quotients of factorials must be bounded; this reduces
        the control by induction to Lemma
        \ref{prop:lem-hard}. Constants which appear must be carefully
        chosen so that the induction can proceed. In particular, given
        a fixed object in an analytic class, it will be useful to
        change the parameters (typically $m,r,R$) in its control,
        while maintaining a fixed norm.
      \end{rem}

The classes $H(m,r,V)$ of real-analytic functions introduced in Section
\ref{sec:holom-extens} contain all holomorphic functions. In a similar
manner, the symbol classes $S^{r,R}_m$ contain all classical analytic symbols in
the sense of Sj\"ostrand \cite{sjostrand_singularites_1982}.
\begin{prop}\label{prop:symb-change-var}
  Let $U$ be an open set of $\C^n$ and let $a=(a_k)_{k\geq 0}$ be a
  sequence of
  bounded holomorphic functions on $U$ such that there exists $C>0$
  and $R>0$
  satisfying, for all $k\geq 0$,
  \[
    \sup_U|a_k|\leq CR^kk!.
  \]
  Then for every $V\subset\subset U$ there exists $r>0$ such that
  $a\in S^{r,R}_0(V)$.

  In particular, given an analytic symbol $a$  and an analytic change
  of variables
  $\kappa$, then $a\circ \kappa$ is an analytic symbol.
\end{prop}
\begin{proof}
  By Proposition \ref{prop:Cauchy}, there exists $C_1>0$ and $r>0$
  such that, for every $k\geq 0$, one has $a_k\in H(r,0,V)$ with
  \[
    \|a_k\|_{H(0,r,V)}\leq C_1\sup_U|a_k|.
  \]
  In other terms, for every $k\geq 0,j\geq 0$, one has
  \[
    \|a_k\|_{C^{j}(V)}\leq C_1Cr^jR^kj!k!\leq C_1Cr^jR^k(j+k)!.
  \]
  Hence $a\in S^{r,R}_0(V)$.

  By a power series expansion, an analytic symbol $a$ satisfies, on some holomorphic
  extension $U$ of its domain of definition,
  \[
    \sup_U|\tilde{a}_k|\leq CR^kk!.
  \]
  This control is not affected by application of the biholomorphism
  $\tilde{\kappa}$, so that
  \[
    \sup_{\tilde{\kappa}^{-1}(U)}|\tilde{a}_k\circ \tilde{\kappa}|\leq
      CR^kk!.
    \]
    By the lines above, $a\circ\kappa$ is an analytic symbol.
\end{proof}

\subsection{Complex stationary phase lemma}
\label{sec:compl-stat-phase}
In this subsection we present the tools of stationary phase in the
context of real-analytic regularity, as developed by Sj\"ostrand \cite{sjostrand_singularites_1982}.
We wish to study integrals of the form

\[
  \int_{\Omega}e^{N\Phi(x)}a(x)\dd x,
\]
as $N\to +\infty$. If $\Phi$ is purely
imaginary, then by integration by parts, this
integral is $O(N^{-\infty})$ away from the points where $d\Phi$
vanishes. At such points, if $\Phi$ is Morse, a change of variables
leads to the usual case where $\Phi$ is quadratic nondegenerate; then
there is a full expansion of the integral in decreasing powers of
$N$. If $\Phi$ is real-valued, a similar analysis (Laplace method)
yields a related expansion.

On one hand, we wish to study such an integral, in the more general case where $i\Phi$
is complex-valued. On the other hand we want to improve the
$O(N^{-\infty})$ estimates into $O(e^{-cN})$. This is done via a
complex change of variables; to this end we have to impose real-analytic
regularity on $\Phi$ and $a$.

Let us introduce a notion of analytic phase, which generalises
positive phase functions as appearing in \cite{sjostrand_singularites_1982}.

\begin{defn}\label{def:phasefunc}
  Let $d,k\in \N$. Let $\Omega$ be a domain of $\R^d$.
  Let $\Phi$ be
  a real-analytic function on $\Omega\times \R^k$. For each
  $\lambda\in \R^k$ we let $\Phi_{\lambda}=\Phi(\cdot,\lambda)$. Then $\Phi$ is said to be an \emph{analytic phase} on $\Omega$ under the following
  conditions.
  \begin{itemize}
  \item There exists an open set $\widetilde{\Omega}\subset \C^d$ such
    that, for every $\lambda\in \R^k$, the function $\Phi_{\lambda}$ extends to a holomorphic
    function $\widetilde{\Phi}_{\lambda}$ on 
    $\widetilde{\Omega}$.
  \item For every $\lambda\in \R^k$, there exists exactly one point $\tilde{x}_{\lambda}\in \widetilde{\Omega}$
    such that $d\widetilde{\Phi}_{\lambda}(\tilde{x}_{\lambda})=0$;
    this critical point is non-degenerate, with
    $\widetilde{\Phi}_{\lambda}(\widetilde{x}_{\lambda})=0$.
  \item One has $\widetilde{x}_0=0$ and moreover $\Re\Phi_0< 0$ on
    $\Omega\setminus \{0\}.$
  \end{itemize}
\end{defn}

Under the conditions of Definition \ref{def:phasefunc}, the function
$\lambda\mapsto \widetilde{x}_{\lambda}$ is real-analytic.

We now recall the stationary phase lemma in analytic regularity.
\begin{prop}\label{prop:HSPL}\cite{sjostrand_singularites_1982,hitrik_two_2013}
  Let $\Phi$ be an analytic phase on a domain $\Omega$. There exists
  $c>0,c'>0,C'>0$, a neighbourhood $\Lambda\subset \R^k$ of zero, and
  a biholomorphism $\widetilde{\kappa}_{\lambda}$, with real-analytic
  dependence\footnote{By this we mean: a real-analytic function
    $\kappa$ on $U\times \Lambda$, where $U$ is a
    neighbourhood of $0$ in $\widetilde{\Omega}$, holomorphic in the
    first variable, such that there exists $\sigma$ with
    the same properties, satisfying
    $\sigma(\kappa(x,\lambda),\lambda)=\kappa(\sigma(x,\lambda),\lambda)=x$
  for all $(x,\lambda)\in U\times \Lambda$.}
  on $\lambda\in \Lambda$, such that the associated Laplace operator
  $\widetilde{\Delta}(\lambda)=\kappa_{\lambda}\circ\Delta\circ\kappa_{\lambda}^{-1}$ satisfies, for every
  function $a_{\lambda}$ holomorphic on $\widetilde{\Omega}$:
  \[
    \int_{\Omega}e^{N\Phi_{\lambda}}a_{\lambda}=\sum_{k=0}^{cN}\left(k!N^{\frac
        d2+k}\right)^{-1}\widetilde{\Delta}(\lambda)^k(\widetilde{a}_{\lambda}J^{-1}_{\lambda})(\widetilde{x}_{\lambda})+R_{\lambda}(N),
  \]
  where, uniformly in $\lambda\in \Lambda$,
  \[
    |R_{\lambda}(N)|\leq
    Ce^{-c'N}\sup_{\widetilde{\Omega}}|\widetilde{a}_{\lambda}|,
  \]
  and $J_{\lambda}$ is the Jacobian determinant associated with the
  change of variables. 
\end{prop}


\section{Calculus of covariant Toeplitz operators}
\label{sec:szego-kernel-general}

In this section we prove our three main theorems.

We begin in Subsection \ref{sec:covar-toepl-oper} with the definition,
and the first properties, of covariant Toeplitz operators. Then, in
Subsections \ref{sec:study-phase-function-1} to
\ref{sec:invers-covar-toepl}, we study them. We prove that they can be
composed (Proposition \ref{prop:cov-compo}), and inverted (Propositions
\ref{prop:cov-inv} and \ref{prop:pre-inv}), with a precise control on the analytic classes involved. This
allows us to prove Theorem \ref{thr:Szeg-gen}: see the beginning of Section \ref{sec:invers-covar-toepl} for
a detailed proof strategy for Theorems \ref{thr:Szeg-gen} and \ref{thr:Compo}.
To conclude, in Subsection \ref{sec:expon-decay-low} we
prove Theorem \ref{thr:Gen-exp-decay}.

Until the end of Section \ref{sec:szego-kernel-general}, $M$ is a compact
real-analytic quantizable K\"ahler manifold of dimension $d$.

\subsection{Covariant Toeplitz operators}
\label{sec:covar-toepl-oper}

\begin{defn}\label{def:Anal-Toep}
 Let $U$ denote a small, smooth neighbourhood of the codiagonal in
  $M\times M$; for instance $U=\{(x,y)\in
  M\times M, \dist(x,y)<\epsilon\}$ with $\epsilon$ small enough so that
  the section $\Psi^N$ of Definition \ref{def:natural-section} is
  defined on a neighbourhood of $U$.
The space $T^{-,r,R}_m(U)$ of \emph{covariant analytic
        Toeplitz operators} consists of operators with kernel
      \[
        T_N^{cov}(f):(x,y)\mapsto N^{d}\mathbb{1}_{(x,y)\in U}\Psi^N(x,y)f(N)(x,y),
      \]
      where $f(N)$ is the summation of an analytic symbol $f\in
      S^{r,R}_m(U)$, with $f$ holomorphic in the first variable and
      anti-holomorphic in the second variable.
    \end{defn}

    \begin{rem}\label{prop:cov-close-H_N}
      Since $\Psi^N$ is exponentially small near the boundary of $U$,
      by Proposition \ref{prop:Kohn}, the image of a covariant
      Toeplitz operator is exponentially close to its projection on
      $H^0(M,L^{\otimes N})$.
    \end{rem}

\subsection{Study of an analytic phase}
\label{sec:study-phase-function-1}

In this work, covariant Toeplitz operators of Definition
\ref{def:Anal-Toep} have the following integral kernels:
\[
  T_N^{cov}(f):(x,y)\mapsto
  \Psi^N(x,y)\left(\sum_{k=0}^{cN}N^{d-k}f_k(x,y)\right).
\]
The integral kernel of the composition of two covariant Toeplitz  is of particular
interest, so let us study its phase. 
 
If $f$ and $g$ are analytic symbols, then
$T_N^{cov}(f)T_N^{cov}(g)$ has the following kernel:
\begin{multline*}
  (x,z)\mapsto
  \Psi^N(x,z)\int_Me^{N(2\widetilde{\phi}(x,y)-2\phi(y)+2\widetilde{\phi}(y,z)-2\widetilde{\phi}(x,z))}\\
  \times \left(\sum_{k=0}^{cN}N^{d-k}f_k(x,y)\right)\left(\sum_{j=0}^{cN}N^{d-j}g_j(y,z)\right)\dd
  y.
\end{multline*}
We let $\Phi_1$ be the complex extension (with respect to the middle
variable) of the phase appearing in the last formula:
\[
  \Phi_1:(x,y,\overline{w},\overline{z})\mapsto
  2\widetilde{\phi}(x,\overline{w})-2\widetilde{\phi}(y,\overline{w})+2\widetilde{\phi}(y,\overline{z})-2\widetilde{\phi}(x,\overline{z}).
\]
We write $\Phi_1(x,y,\overline{w},\overline{z})$ to indicate
anti-holomorphic dependence on the two last variables. In
particular, $\Phi_1$ is holomorphic on the open set $U\times U$ of $M\times
\widetilde{M}\times \overline{M}=M_x\times (M_y\times
\overline{M}_{\overline{w}})\times \overline{M}_{\overline{z}}$.

The fact that $\Phi_1$ is a well-behaved phase function is well-known;
let us state it in this real-analytic context.
\begin{prop}
  \label{prop:Phase}
  There exists a smooth neighbourhood $U$ of \[\{(x,\overline{z}) \in
  M\times \overline{M}, \overline{x}=\overline{z}\}\] such that
  function $\Phi_1$, on the open set \[\{(x,y,\overline{y},\overline{z}),
  (x,\overline{w})\in U, (y,\overline{w})\in U,(x,\overline{z})\in U\},\] is an analytic phase of $(y,\overline{w})$, with parameter $\lambda=(x,\overline{z})$. The
  critical point is $(x,\overline{z})$.

  In particular, after a trivialisation of a tubular neighbourhood of
  \[\{(x,y,\overline{w},\overline{z})\in M\times\widetilde{M}\times
    \overline{M}, (x,\overline{z})\in U, (y,\overline{w})=(x,\overline{z})\}\] in \[\{(x,y,\overline{w},\overline{z})\in
  M\times \widetilde{M}\times \overline{M},(x,\overline{z})\in U\}\] as a vector bundle over the former, the
  analytic phase $\Phi_1$ satisfies the assumptions of Definition \ref{def:phasefunc}.
\end{prop}

\subsection{Composition of covariant Toeplitz operators}
\label{sec:comp-covar-toepl}

In this subsection we study the composition rules for operators with
kernels of the form

\[
  T_N^{cov}(f)(x,y)=\Psi^N(x,y)\left(\sum_{k=0}^{cN}N^{d-k}f_k(x,y)\right).
\]
Here, for a small, smooth neighbourhood $U$ of the diagonal in
$M\times M$, one has $f\in S^{r,R}_m(U)$, and $f$ is
holomorphic in the first variable and anti-holomorphic in the second
variable.

Such operators can be formally composed, that is, there holds
\[T_N^{cov}(f)T_N^{cov}(g)=T_N^{cov}(f\sharp g)+O(N^{-\infty})\] where
$f\sharp g$ is a classical symbol. This formal calculus satisfies a
Wick rule (Proposition \ref{prop:bound-nb-deriv}). This allows us, in Proposition
\ref{prop:cov-compo}, to prove that, if $f$ and $g$ are analytic
symbols, then $f\sharp g$ is also an analytic symbol, so that one can
perform an analytic summation (as in Proposition
\ref{prop:sum-anal-symb}), and the error in the composition becoms $O(e^{-cN})$.

\begin{prop}\label{prop:bound-nb-deriv} (See also
  \cite{charles_aspects_2000}, Lemmes 2.33 and following)
  The composition of two covariant Toeplitz operators can be written
  as a formal series in $N^{-1}$. More precisely, if $f$
  and $g$ are functions on a neighbourhood of the diagonal in $M\times
  M$, holomorphic in the first variable, anti-holomorphic in the
  second variable, then
  \[
    T_N^{cov}(f)T_N^{cov}(g)=T_N^{cov}(f\sharp g)+O(N^{-\infty}),
  \]
  where $f\sharp g$ is a formal series $h\sim \sum_{k\geq
    0}N^{-k}(f\sharp g)_k$,
  holomorphic in the first variable, anti-holomorphic in the second variable. The
  composition law can be written as
  \[
    (f\sharp g)_k=B_k(f,g),
  \]
  where $B_k$ is a bidifferential operator of degree at most $k$ in f
  and at most $k$ in $g$.
\end{prop}
Since the thesis \cite{charles_aspects_2000} is in French, we present
a self-contained proof of this fact in the appendix.

\begin{rem}
Proposition \ref{prop:bound-nb-deriv} is not a trivial consequence of the expression
of the phase $\Phi_1$. Indeed, the composition rule for $f\sharp g$ consists in
finding, for $x,\overline{z}$ fixed, a holomorphic change of variables
$(y,\overline{w})\mapsto (v,\overline{v})$ such that $\Phi_1=-v\cdot
\overline{v}$; then, if $J$ denotes the Jacobian of this change of variables,
\begin{multline*}
  (f\sharp
  g)_k(x,\overline{z})\\=\sum_{n=0}^k\cfrac{\partial_{\overline{v}}^n\partial_v^n}{n!}\left(\sum_{l=0}^{k-n}f_l(x,\overline{w}(x,v,\overline{v},\overline{z}))g_{k-n-l}(y(x,v,\overline{v},\overline{z}),\overline{z})J(x,v,\overline{v},\overline{z})\right)_{v=\overline{v}=0}.
\end{multline*}
If the Morse change of variables can be split as
$(y,\overline{w})\mapsto (v(y),\overline{v}(\overline{w}))$, then
Proposition \ref{prop:bound-nb-deriv} follows immediately from the
formula above: holomorphic derivatives (in $v$) will only hit $g_{k-n-l}$ or $J$
while anti-holomorphic derivatives (in $\overline{v}$) will only hit
$f_{l}$ or $J$, so that both $f$ and $g$ are differentiated at
most $k$ times.

In the model case where $\Phi_1$ is a quadratic form (such as on
Bargmann space), such a splitting is indeed true. However, the
property $\Phi_1=-v(y)\cdot
\overline{v}(\overline{w})$ is, in dimension 1, equivalent to the easily checked identity
\[
  \partial_y\partial_{\overline{w}}\log(\Phi_1)=0.
\]
This property is false, for instance, if $\Phi_1$ represents the usual
$2$-sphere in the stereographic projection. In this case,
\begin{multline*}
  \Phi_1:(x,y,\overline{w},\overline{z})\mapsto 2(\log(1+x\cdot
    \overline{w})+\log(1+y\cdot \overline{z})\\-\log(1+y\cdot
    \overline{w})-\log(1+x\cdot \overline{z})).
\end{multline*}
Restricting to $x=z=0$, we obtain
  \[
    \partial_y\partial_{\overline{w}}\log(\Phi_1)=\frac{-2}{(1+y\cdot
      \overline{w})\log(1+y\cdot \overline{w})}+\frac{2y\cdot
      \overline{w}\left(\log(1+y\cdot
        \overline{w})+1\right)}{\left((1+y\cdot
        \overline{w})\log(1+y\cdot \overline{w})\right)^2},
  \]
  which is obviously non-zero (in particular, it is equal to
  $\frac{1-\log(2)}{2\log^2(2)}$ at $y=\overline{w}=1$.)
\end{rem}

Proposition \ref{prop:bound-nb-deriv} predicts that, when applying a stationary
phase lemma to $\Phi_1$ in order to study $T_N^{cov}(f)T_N^{cov}(g)$,
at order $k$, only derivatives of $f$ and $g$ at order $k$ will
appear. 
However, in the stationary phase
(Lemma \ref{prop:HSPL}), these derivatives appear in the form of an usual Laplace operator, conjugated by a
change of variables.
Before proceeding further, let us prove a technical lemma.
\begin{lem}\label{lem:ctrl-cht-var}
  Let $U,V,\Lambda$ be domains in $\C^d$ containing $0$. Let $\kappa_{\lambda}$ be a biholomorphism from $V$ to $U$,with real-analytic dependence on
  $\lambda\in \Lambda$, and such that $\kappa_{\lambda}(0)=0$ for all
  $\lambda\in \Lambda$. Let $\kappa(\lambda,v)\mapsto
  \kappa_{\lambda}(v)$, and suppose that there exists $C_{\kappa},r_0,m_0$ such
  that, for all $j\in \N_0$, one has
  \[
    \|\kappa\|_{C^j(V\times \Lambda)}\leq
    C\cfrac{r_0^{j}j!}{(j+1)^{m_0}}.
    \]
  Then the
  following is true for all $m\geq m_0,r\geq 8r_02^{m-m_0}$.

  Let $f$ be a real-analytic function on $U\times \Lambda$, and
  suppose that there exists $C_f$ and $k\geq 0$ such that
  \[
    \|f\|_{C^{j}(U\times \Lambda)}\leq C_f\cfrac{r^j(j+k)!}{(j+k+1)^m}.
  \]
  Let $n\leq k$ and $i\leq 2n$; let $\nabla^{i}_v$ denote the
  $i$-th gradient (as in Definition \ref{def:Anal-fct-space}) over the first set of variables, acting
  on $V\times \Lambda$; then \[g\mapsto (\lambda\mapsto
    \nabla^i_vg(\kappa_{\lambda}(v),\lambda)_{v=0})\] is a
  differential operator of degree $i$, from functions on $U\times
  \Lambda$ to vector-valued functions on
  $\Lambda$. Let $(\nabla^i_{\kappa})^{[\leq n]}$ denote the
  truncation of this differential operator to a differential operator
  of degree less than $n$.

  Then, with
  \[
    \gamma=4Cr,
  \]
  one has, for every $j\geq 0$, 
  \begin{multline*}
    \|(\nabla^i_{\kappa})^{[\leq n]}f\|_{\ell^1(C^j(\Lambda))}\leq
    i^{d+1}j^{d+1}\gamma^iC_f\frac{r^{j+i}}{(i+j+l+1)^m}\\ \times\begin{cases}(i+j+k)!&\text{if
        $i\leq
        n$}\\\max((n+j+k)!(i-n)!,(j+k)!i!)&\text{otherwise.}\end{cases}
  \end{multline*}
  
  \end{lem}
  \begin{proof}
    Let us make explicit the operator $(\nabla^i_{\kappa})^{[\leq n]}$.
     Given a polyindex $\mu$ with $|\mu|=i$, the Faà di Bruno formula
  states:
  \[
    \partial^{\mu}_v(f(\kappa_{\lambda}(v),\lambda))_{v=0}=\sum_{P\in
      \Pi(\{1,\ldots,i\})}f^{|P|}(0,\lambda)\prod_{E\in
      P}(\partial^E\kappa_{\lambda})(0),
  \]
  where the sum runs among all partitions $P=\{E_1,\ldots,E_{|P|}\}$
  of $\{1,\ldots,i\}$
  .
  
  When considering the operator $(\nabla^{i}_{\kappa})^{[\leq n]}$, we
  only need to consider partitions $P$ such that $|P|\leq n$. If the sizes
  $|E_1|=s_1,\ldots, |E_{|P|}|=s_{|P|}$ of the elements of $P$ are
  fixed, the number of possible partitions is simply
  \[
    \cfrac{i!}{(|P|)!s_1!\ldots s_{|P|}!}.
    \]
    Then, since there are less than $i^d$ polyindices $\mu$ with
    $|\mu|=i$, one has, for all $\rho\in \N_0^d$ with $|\rho|=j$, by
    differentiation of the Faà di Bruno formula and
    Proposition \ref{prop:prod-Cj},
    \begin{multline*}
      \|\partial^{\rho}((\nabla^i_{\kappa})^{[\leq n]}f)\|_{\ell^1}\leq \\i^d\sum_{|P|=1}^{\min(n,i)}\hspace{-5pt}
    \sum_{\substack{e_0+\ldots+e_{|P|}=j\\s_1+\ldots+s_{|P|}=|P|}}\cfrac{j!}{e_0!e_1!\ldots e_{|P|}!}\,\cfrac{i!}{(|P|)!s_1!\ldots s_{|P|}!}\,\|f\|_{C^{|P|+e_0}}\prod_{i=1}^{
      |P|}\|\kappa\|_{C^{s_i+e_i}}.
  \
\end{multline*}
Here $\kappa$ denotes the real-analytic function $(\lambda,v)\mapsto
  \kappa_{\lambda}(v)$.
  
  In particular, since there are less than $j^d$ polyindices $\rho$ such that
  $|\rho|=j$, one has
  \begin{multline}\label{eq:Faa}
   \|\partial^{\rho}((\nabla^i_{\kappa})^{[\leq n]}f)\|_{\ell^1}\leq i^dj^d\\\times\sum_{|P|=1}^{\min(n,i)}\,
    \hspace{-0.7em}\sum_{\substack{e_0+\ldots+e_{|P|}=j\\s_1+\ldots+s_{|P|}=|P|}}\cfrac{j!}{e_0!e_1!\ldots e_{|P|}!}\,\cfrac{i!}{(|P|)!s_1!\ldots s_{|P|}!}\,\|f\|_{C^{|P|+e_0}}\prod_{i=1}^{
      |P|}\|\kappa\|_{C^{s_i+e_i}}.
  \end{multline}
  Since, for all $j\geq 0$, one has
  \[
    \|\kappa\|_{C^j(V\times \Lambda)}\leq
    C\cfrac{r_0^{j}j!}{(j+1)^{m_0}},
  \]
  by Lemma \ref{prop:move-m-r}, for all $m\geq m_0, r\geq
  8r_02^{m-m_0}$, one has
  \[
    \|\kappa\|_{C^j}\leq C\cfrac{(r/8)^jj!}{(j+1)^m}.
  \]
  In particular, if $j\geq 1$, there holds
  \[
    \|\kappa\|_{C^j}\leq
    C\frac{(r/4)^j(j-1)!}{j^m}j\left(\frac{j}{j+1}\right)^m2^{-j}\leq C\frac{(r/4)^j(j-1)!}{j^m},
  \]
  since
  \[
    j\left(\frac{j}{j+1}\right)^m2^{-j}\leq j2^{-j}\leq 1.
  \]
  Let us suppose further that
  \[
    \|f\|_{C^j(U\times \Lambda)}\leq C_f\cfrac{r^jR^l(j+l)!}{(j+l+1)^m}.
    \]
  Then, the contribution of one term in the sum \eqref{eq:Faa} is
  \begin{multline*}
    \cfrac{j!}{e_0!e_1!\ldots e_{|P|}!}\,\cfrac{i!}{(|P|)!s_1!\ldots s_{|P|}!}\,\|f\|_{C^{|P|+e_0}}\prod_{i=1}^{
      |P|}\|\kappa\|_{C^{s_i+e_i}}\\
    \leq
    C_fC^{|P|}\cfrac{r^{|P|+e_0}(r/4)^{i+j-e_0}R^l(|P|+e_0+l)!i!}{(|P|+e_0+l+1)^m(|P|)!s_1!\ldots
      s_{|P|}!}\\ \times\cfrac{j!(s_1+e_1-1)!\ldots(s_{|P|}+e_{|P|}-1)!}{e_0!e_1!\ldots
      e_{|P|}!(s_1+e_1)^{m}\ldots(s_{|P|}+e_{|P|})^m}.
  \end{multline*}
  As $e_0+\ldots+e_{|P|}=j$ and $s_1+\ldots+s_{|P|}=i$, and
  since, as soon as $x\geq 0$, $y\geq 0$, there holds
  \[
    (1+x)(1+y)=1+x+y+xy\geq 1+x+y,
  \]
    one has
  \begin{multline*}
    (|P|+e_0+l+1)^m(s_1+e_1)^m\ldots(s_{|P|}+e_{|P|})^m\\\geq
    (|P|+j+i+l-|P|+1)^m=(j+i+l+1)^m,
  \end{multline*}
  so that one can simplify
  \begin{multline*}
    C_fC^{|P|}\cfrac{r^{|P|+e_0}(r/4)^{i+j-e_0}R^l(|P|+e_0+l)!i!}{(|P|+e_0+l+1)^m(|P|)!s_1!\ldots
      s_{|P|}!}\\ \times\cfrac{j!(s_1+e_1-1)!\ldots(s_{|P|}+e_{|P|}-1)!}{e_0!e_1!\ldots
      e_{|P|}!(s_1+e_1)^{m}\ldots(s_{|P|}+e_{|P|})^m}\\\leq
    C_fC^{|P|}\cfrac{r^{|P|+e_0}(r/4)^{i+j-e_0}R^l(|P|+e_0+l)!}{(j+i+l+1)^m}\\
    \times \cfrac{i!j!(s_1+e_1-1)!\ldots(s_{|P|}+e_{|P|}-1)!}{e_0!(|P|)!s_1!\ldots s_{|P|}!e_1!\ldots
      e_{|P|}!}.
  \end{multline*}
  By Lemma \ref{lem:binom-multi}, one has
  \[
    \cfrac{(s_1+e_1-1)!\ldots(s_{|P|}+e_{|P|}-1)!}{s_1!\ldots s_{|P|}!e_1!\ldots
      e_{|P|}!}\leq
    \cfrac{(i-|P|+j-e_0)!}{(i-|P|+1)!(j-e_0)!}.
    \]

Hence, the contribution of one term in the sum \eqref{eq:Faa} is smaller than
\begin{multline*}
  C_fC^{|P|}\cfrac{i!}{(|P|)!(i-|P|+1)!}\\ \times\cfrac{r^{|P|+e_0}(r/4)^{i+j-e_0}R^l(|P|+e_0+l)!j!(i-|P|+j-e_0)!}{(j+i+l+1)^me_0!(j-e_0)!}.
\end{multline*}
As $(i-|P|+j-e_0)!\leq (j-e_0)!(i-|P|)!2^{i+j-e_0}$ and
$i!\leq 2^{i}(|P|)!(i-|P|)!$, we control
each term in the sum \eqref{eq:Faa} with
\begin{multline*}
  C_f2^{e_0-j}C^{|P|}r^{i}\cfrac{r^{j+|P|}R^l(|P|+e_0+l)!}{(j+i+l+1)^m}\,\cfrac{j!(i-|P|)!}{e_0!}
  \\ \leq
   C_f2^{e_0-j}(Cr)^{i}\cfrac{r^{j+i}R^l(|P|+e_0+l)!}{(j+i+l+1)^m}\,\cfrac{j!(i-|P|)!}{e_0!}.
 \end{multline*}
 
There are $\binom{i}{|P|}\leq 2^{i}$ choices for positive $s_1,\ldots,s_{|P|}$ such that their sum is $i$;
similarly, there are $\binom{j-e_0+|P|}{|P|}\leq
2^{j-e_0+|P|}$ choices for non-negative $e_1,\ldots,e_{|P|}$ such that their sum is $j-e_0$.
Hence
\begin{align*}
  &\|(\nabla^i_{\kappa})^{[\leq n]}f\|_{\ell^1(C^{j})}\\ &\leq
  i^dj^d\sum_{|P|=1}^{\min(n,i)}\sum_{e_0=0}^{j}2^{j+|P|-e_0}2^{i}C_f2^{e_0-j}(Cr)^{i}\cfrac{r^{j+i}R^l(|P|+e_0+l)!}{(j+i+l+1)^m}\,\cfrac{j!(i-|P|)!}{e_0!}\\
  &\leq i^dj^d\sum_{|P|=1}^{\min(n,i)}\sum_{e_0=0}^{j}C_f(4Cr)^{i}\cfrac{r^{j+i}R^l(|P|+e_0+l)!}{(j+i+l+1)^m}\,\cfrac{j!(i-|P|)!}{e_0!}.
\end{align*}

The terms in the sum above are increasing with respect to $e_0$, so that
\begin{multline*}
  \|\nabla_{v}^{i}f(x,\kappa(x,v,\overline{z}))_{v=0}\|_{\ell^1(C^{j})}\\\leq
  i^dj^{d+1}\sum_{|P|=1}^{\min(n,i)}
  C_f(4Cr)^{i}\cfrac{r^{j+i}R^l(|P|+j+l)!}{(i+j+l+1)^m}\,(i-|P|)!.
\end{multline*}
Observe that the quantity in the sum above
is log-convex with respect to $|P|$ as it is a product of factorials, so that
\begin{multline*}
  \|(\nabla^i_{\kappa})^{[\leq n]}f\|_{\ell^1(C^{j})}\\\leq
  i^{d+1}j^{d+1}C_f\cfrac{r^{j+i}R^l}{(i+j+l+1)^m}(4Cr)^{i}\max((n+j+l)!(i-n)!\;,\;(j+l)!i!)
\end{multline*}
if $i\geq n$, and
\[
  \|(\nabla^i_{\kappa})^{[\leq n]}f\|_{\ell^1(C^{j})}\leq
  i^{d+1}j^{d+1}
  C_f\cfrac{r^{j+i}R^l}{(i+j+l+1)^m}(4Cr)^{i}(i+j+l)!
\]
if $i\leq n$.
This concludes the proof, with $\gamma=4Cr$.
  \end{proof}

We are in position to prove the first part of Theorem
\ref{thr:Compo}, which does not use the structure of the Bergman
kernel. Let us prove that the composition
of two covariant Toeplitz operators with analytic symbols also admits an analytic symbol,
up to an exponentially small error.
\begin{prop}
  \label{prop:cov-compo}
  There exists a small neighbourhood $U$ of the diagonal in $M\times
  M$, and constants $C,m_0,r_0$ such that, for every $m\geq
  m_0,r\geq r_0,R\geq Cr^3$, there exists $c'>0$ such that, for every $f\in S^{r,R}_m(U)$ and $g\in
  S^{2r,2R}_m(U)$, holomorphic in the first variable, anti-holomorphic
  in the second variable, there exists $f\sharp g\in S^{2r,2R}_m(U)$
  with the same properties, such that
  \[
    \|T_N^{cov}(f)T_N^{cov}(g)-T_N^{cov}(f\sharp g)\|_{L^2\mapsto
      L^2}\leq Ce^{-c'N}\|g\|_{S_m^{2r,2R}(U)}\|f\|_{S_m^{r,R}(U)}.
  \]
  Moreover
  \[
    \|f\sharp g\|_{S_m^{2r,2R}(U)}\leq
    C\|g\|_{S_m^{2r,2R}(U)}\|f\|_{S_m^{r,R}(U)}.
  \]
\end{prop}
\begin{proof}
  The kernel of $T_N^{cov}(f)T_N^{cov}(g)$ can be written as
  \begin{multline*}
    (x,z)\mapsto \Psi^N(x,{z})\int_{y\in
      M}e^{N\Phi_1(x,y,\overline{y},z)}\left(\sum_{k=0}^{cN}N^{d-k}f_k(x,\overline{y})\right)\\
    \times\left(\sum_{j=0}^{cN}N^{d-j}g_j(y,\overline{z})\right)\dd
    y.
  \end{multline*}
  Here, and until the end of the proof, we write $f_k(x,\overline{y})$ to indicate that $f_k$
  is holomorphic in the first variable and anti-holomorphic in the
  second variable. We similarly write $g_j(y,\overline{z})$.

  Since $\Phi_1$ is an analytic phase (Proposition \ref{prop:Phase}), let us apply the
  stationary phase lemma (Proposition \ref{prop:HSPL}). There exists a
  biholomorphism on a neighbourhood of $x$ in $\widetilde{M}$, of the form
  \[
    \kappa_{(x,\overline{z})}:(y,\overline{y})\mapsto
    v(x,y,\overline{y},\overline{z}),
  \]with holomorphic dependence on $(x,\overline{z})$ (that is,
  holomorphic in $x$ and anti-holomorphic in $z$), in which the phase $\Phi_1$
  can be written as $-|v|^2$. In particular,
  \[
    v(x,x,\overline{z},\overline{z})=0.
  \]
  Let $J$ denote the
  Jacobian of this change of variables. Then
  \begin{multline*}
    T_N^{cov}(f)T_N^{cov}(g)(x,z)=\\\Psi^N(x,{z})\sum_{k,j,n=0}^{\ldots}N^{d-k-j-n}\cfrac{\Delta^n_v}{n!}\left(f_k(x,\overline{y}(x,v,\overline{z}))g_j(y(x,v,\overline{z}),\overline{z})J(x,v,\overline{z})\right)_{v=0}\\+\ldots
  \end{multline*}
  We will make sense of this sum later on; that is, prove that one can
  sum until $k,j$ and $n$ are equal to $cN$, up to an exponentially
  small error. For the moment, let us treat this formula in decreasing
  powers of $N$. Writing
  \[
    T_N^{cov}(f)T_N^{cov}(g)(x,z)=T_N^{cov}(f\sharp
    g)(x,\overline{z})=\Psi^N(x,{z})\sum_{k=0}^{\cdots}N^{d-k}(f\sharp
    g)_k(x,\overline{z})+\ldots
    \]the symbol
  $f\sharp g$ must be holomorphic in the first variable,
  anti-holomorphic in the second variable, and such that
  \[
    (f\sharp
    g)_k(x,\overline{z})=\sum_{n=0}^k\cfrac{\Delta^n_{v}}{n!}\left(\sum_{l=0}^{k-n}f_l(x,\overline{y}(x,v,\overline{z}))g_{k-n-l}(y(x,v,\overline{z}),\overline{z})J(x,v,\overline{z})\right)_{v=0}.
  \]
  
  Here the Laplace operator acts on $v$.

  The proof proceeds now in two steps. In the first step, we write
  a control of the formal symbol $f\sharp g$ using the analytic symbol
  structure of $f$ and $g$ and Lemma \ref{lem:ctrl-cht-var}. This
  control involves a complicated quotient of factorials as well as a
  rational expression similar to the one appearing in Lemma
  \ref{prop:lem-hard}. The second step is a control the quotients of factorials,
  thus reducing the proof that $f\sharp g\in S^{2r,2R}_m$ to Lemma
  \ref{prop:lem-hard}. It is then standard
  \cite{sjostrand_singularites_1982,martinez_introduction_2002} to
  check that, if $f,g,f\sharp g$ are analytic symbols then one can
  perform an analytic summation to prove that \[T_N^{cov}(f)T_N^{cov}(g)=T_N^{cov}(f\sharp g)+O(e^{-cN}).\]
  
  {\bf First step.}
  
  We wish to control $\|(f\sharp g)_k\|_{C^j(U)}$, which amounts to
  control, for each $0\leq n\leq k, 0 \leq l\leq k-n$, the $C^j$-norm of
  \[
    (x,z)\mapsto \Delta_{v}^n\left(f_l(x,\overline{y}(x,v,\overline{z}))g_{k-n-l}(y(x,v,\overline{z}),\overline{z})J(x,v,\overline{z})\right)_{v=0}.
    \]
  This bidifferential operator acting on $f_l$ and $g_{k-n-l}$
  coincides, up to a multiplicative factor, with the operator $B_n$ considered in Proposition
  \ref{prop:bound-nb-deriv}. Indeed, if $f=f_0$ and $g=g_0$, then
  \[
    (f\sharp g)_k(x,\overline{z})=\cfrac{\Delta^k_{v}}{k!}\left(f_0(x,\overline{y}(x,v,\overline{z}))g_0(y(x,v,\overline{z}))J(x,v,\overline{z})\right)_{v=0}=B_k(f_0,g_0),
    \]
  where $(B_k)_{k\geq 0}$ is the sequence of bidifferential operators appearing
  in Proposition \ref{prop:bound-nb-deriv}.
  In particular, when expanding
\[
    \Delta^n_{v}\left(f_l(x,\overline{y}(x,v,\overline{z}))g_{k-n-l}(y(x,v,\overline{z}),\overline{z})J(x,v,\overline{z})\right)_{v=0},
  \]
  using the Leibniz and Fa\`a di Bruno formulas, no derivative of $f_l$
  and $g_{k-n-l}$ of order greater than $n$ will appear. Let us write
  this expansion.

  Until the end of the proof, $C^j$ or analytic norms of functions are
  implicitly on the domain $U$ or $U\times U$.
  
  For every $n\in \N_0$, by the multinomial formula, there holds
\[
\Delta_{v}^n=\left(\sum_{i=1}^{2d}\cfrac{\partial^2}{\partial v_j^2}\right)^n=\sum_{\substack{\mu\in \N_0^{2d}\\|\mu|=n}}\cfrac{n!}{\mu!}\,\partial_{v}^{2\mu}.
\]
  Applying the generalised Leibniz rule twice, one has then
  \begin{multline*}
    \Delta^n_{v}\left(f_l(x,\overline{y}(x,v,\overline{z}))g_{k-n-l}(y(x,v,\overline{z}),\overline{z})J(x,v,\overline{z})\right)_{v=0}
  \\=\sum_{\substack{|\mu|=n\\\nu_1+\nu_2\leq
      2\mu}}\cfrac{n!(2\mu)!}{\mu!\nu_1!\nu_2!(2\mu-\nu_1-\nu_2)!}\,\partial^{\nu_1}_{v}f_l(x,\overline{y}(x,v,\overline{z}))_{v=0}\\
  \times \partial^{\nu_2}_{v}g_{k-n-l}(y(x,v,\overline{z}),\overline{z})_{v=0}\partial^{2\mu-\nu_1-\nu_2}_{v}J_{v=0}.
\end{multline*}
By Proposition \ref{prop:bound-nb-deriv}, in the formula above one can replace
$\partial_v^{\nu_1}f(x,\overline{y}(x,v,\overline{z}))_{v=0}$ by its truncation  into a differential operator of degree less than $n$, applied
on $f$, which we denote $(\partial_{\kappa}^{\nu_1})^{[\leq
  n]}f(x,\overline{z})$ (similarly as in Lemma \ref{lem:ctrl-cht-var}). Similarly one can replace $\partial_v^{\nu_2}g(y(x,v,\overline{z}),\overline{z})_{v=0}$ by
$(\partial_{\kappa}^{\nu_2})^{[\leq n]}g(x,\overline{z})$.
Then
\begin{multline*}
    \Delta^n_{v}\left(f_l(x,\overline{y}(x,v,\overline{z}))g_{k-n-l}(y(x,v,\overline{z}),\overline{z})J(x,v,\overline{z})\right)_{v=0}
  \\=\sum_{\substack{|\mu|=n\\\nu_1+\nu_2\leq
      2\mu}}\cfrac{n!(2\mu)!}{\mu!\nu_1!\nu_2!(2\mu-\nu_1-\nu_2)!}\,(\partial^{\nu_1}_{\kappa})^{[\leq
    n]}f_l(x,\overline{z})\\ \times(\partial^{\nu_2}_{\kappa})^{[\leq n]}g_{k-n-l}(x,\overline{z})\partial^{2\mu-\nu_1-\nu_2}_{v}J_{v=0},
\end{multline*}
with, by Lemma \ref{lem:control-poly-binom},
\begin{align*}
  \cfrac{n!\mu_1!}{\nu_1!\nu_2!(2\mu-\nu_1-\nu_2)!}&=\cfrac{n!}{\mu!}\,\cfrac{(2\mu)!}{\nu_1!(2\mu-\nu_1)!}\,\cfrac{(2\mu-\nu_1)!}{\nu_2!(2\mu-\nu_1-\nu_2)!}\\
                                                   &\leq
                                                     \cfrac{n!}{\mu!}\,\cfrac{(2n)!}{|\nu_1|!(2n-|\nu_1|)!}\,\cfrac{(2n-|\nu_1|)!}{|\nu_2|!(2n-|\nu_1|-|\nu_2|)!}\\
                                                   &=\cfrac{n!}{\mu!}\binom{2n}{|\nu_1|,|\nu_2|}\leq
                                                     (2d)^n\binom{2n}{|\nu_1|,|\nu_2|}.
\end{align*}
Moreover, applying Proposition \ref{prop:prod-Cj} twice,
\begin{multline*}
  \|(\partial^{\nu_1}_{\kappa})^{[\leq
    n]}f_l(x,\overline{z})(\partial^{\nu_2}_{\kappa})^{[\leq n]}g_{k-n-l}(x,\overline{z})\partial^{2\mu-\nu_1-\nu_2}_{v}J_{v=0}\|_{C^j}\\
  \leq
  \sum_{j_1+j_2\leq
    j}\binom{j}{j_1,j_2}\|(\partial^{\nu_1}_{\kappa})^{[\leq
    n]}f_l(x,\overline{z})\|_{C^{j_1}}\\ \times \|(\partial^{\nu_2}_{\kappa})^{[\leq n]}g_{k-n-l}(x,\overline{z})\|_{C^{j_2}}
  \|\partial^{2\mu-\nu_1-\nu_2}_{v}J_{v=0}\|_{C^{j-j_1-j_2}}.
  \end{multline*}
In particular, using the notation
 $(\nabla^{j}_{\kappa})^{[\leq n]}$ as introduced in Lemma
 \ref{lem:ctrl-cht-var}, one has
\begin{multline*}
    \|n!B_n(f_l,g_{k-n-l})\|_{C^j}\\=\left\|\Delta_{v}^n\left(f_l(x,\overline{y}(x,v,\overline{z}))g_{k-n-l}(y(x,v,\overline{z}),\overline{z})J(x,v,\overline{z})\right)_{v=0}\right\|_{C^j}\\
    \leq
    (2d)^n\sum_{\substack{j_1+j_2\leq j\\i_1+i_2\leq
        2n}}\binom{j}{j_1,j_2}\binom{2n}{i_1,i_2}\|(\nabla^{i_1}_{\kappa})^{[\leq
      n]}f_l(x,\overline{z})\|_{\ell^1(C^{j_1})}\\\times \|(\nabla^{i_2}_{\kappa})^{[\leq
      n]}g_{k-n-l}(x,\overline{z})\|_{\ell^1(C^{j_2})}\|\nabla^{2n-i_1-i_2}_{v}J\|_{\ell^1(C^{j-j_1-j_2})}.
  \end{multline*}
  By Lemma \ref{lem:ctrl-cht-var}, for some $\gamma_r$ depending
  linearly on
  $r$ (but independent of $R,m$), one has
  \begin{multline*}
    \|(\nabla^{i_1}_{\kappa})^{[\leq
      n]}f_l(x,\overline{z})\|_{\ell^1(C^{j_1})}\\\leq
    i_1^{d+1}j_1^{d+1}\|f\|_{S^{r,R}_m}\gamma_r^{i_1}\frac{r^{j_1+i_1}R^l}{(i_1+j_1+l+1)^m}A(i_1,j_1,l,n),
  \end{multline*}

  and
  \begin{multline*}
    \|(\nabla^{i_2}_{\kappa})^{[\leq
      n]}g_{k-n-l}(x,\overline{z})\|_{\ell^1(C^{j_2})}\\\leq i_2^{d+1}j_2^{d+1}\|g\|_{S^{2r,2R}_m}\gamma_r^{i_2}\frac{(2r)^{j_2+i_2}(2R)^{k-n-l}}{(i_2+j_2+l+1)^m}A(i_2,j_2,k-n-l,n),
  \end{multline*}
  where
  \[
    A(i,j,l,n)=\begin{cases}(i+j+l)! & \text{if $i\leq n$},\\
      \max((n+j+l)!(i-n)!, (j+l)!i!)& \text{otherwise},\end{cases}
  \]
  
  The real-analytic function $J$ belongs to some fixed analytic space,
  so that there exists $r_0,m_0$ such that.
  \[
    \|J\|_{C^j}\leq C_J\cfrac{r_0^jj!}{(j+1)^{m_0}},
    \]
    If $r\geq 2r_02^{m-m_0}$, by Proposition \ref{prop:move-m-r}, one has
    \[
      \|J\|_{C^j}\leq C_J\cfrac{(r/2)^jj!}{(j+1)^m},
    \]
    hence
\begin{multline*}
  \|(f\sharp g)_k\|_{C^j}\leq\\
  C_J\|f\|_{S^{r,R}_m}\|g\|_{S^{2r,2R}_m}\cfrac{(2r)^j(2R)^k(j+k)!}{(k+j+1)^m}\sum_{n=0}^k\left(\cfrac{\gamma_r
        r^2}{R}\right)^{\hspace{-0.2em}n\hspace{0.2em}}\sum_{l=0}^{k-n}\sum_{i_1+i_2\leq 2n}\sum_{j_1+j_2\leq
      j}\\
\cfrac{(2n)!j!A(i_1,j_1,l,n)A(i_2,j_2,k-l,n)(2n+j-j_1-j_2-i_1-i_2)!}{2^{2n+j-j_1-j_2-i_1-i_2}2^{j_1+i_1+l}i_1!i_2!j_1!j_2!(2n-i_1-i_2)!(j-j_1-j_2)!n!(k+j)!}
    \\\cfrac{i_1^di_2^dj_1^dj_2^d(k+j+1)^m}{(j_1+i_1+l+1)^m(j_2+i_2+k-n-l+1)^m(j+2n-i_1-i_2-j_1-j_2+1)^m}.
  \end{multline*}

  {\bf Second step.}
  
  Let us control the quotient of factorials above. There holds
\[
  \cfrac{(2n+j-j_1-j_2-i_1-i_2)!}{2^{2n+j-j_1-j_2-i_1-i_2}(j-j_1-j_2)!(2n-i_1-i_2)!}=\cfrac{\binom{2n+j-j_1-j_2-i_1-i_2}{j-j_1-j_2}}{2^{2n+j-j_1-j_2-i_1-i_2}}\leq
  1.
\]
Thus, the middle line in the control on $\|(f\sharp g)_k\|_{C^j}$ is
smaller than
\[
  \cfrac{(2n)!j!A(i_1,j_1,l,n)A(i_2,j_2,k-l,n)}{2^{j_1+i_1+l}i_1!i_2!j_1!j_2!n!(k+j)!}.
\]
Let us prove that, if $i_1\leq 2n$, $i_2\leq 2n$, $0\leq l \leq k-n$,
$j_1+j_2\leq j$, then
\[
  \cfrac{(2n)!j!A(i_1,j_1,l,n)A(i_2,j_2,k-l,n)}{2^{j_1+i_1+l}i_1!i_2!j_1!j_2!n!(k+j)!}\leq
  4^n.
\]
For the moment, let us focus on the $i_1\leq n,i_2\leq n$ case.
As $i_1\geq 0$ one has $\frac{1}{2^{i_1}}\leq 1$ and it
remains to control
\[
  \cfrac{(2n)!j!(j_1+i_1+l)!(j_2+i_2+k-n-l)!}{2^{j_1+l}i_1!i_2!j_1!j_2!n!(k+j)!}.
\]
This expression is increasing with respect to $i_1$ and $i_2$, so that
we only need to control the $i_1=i_2=n$ case, which is
\[
  \cfrac{(2n)!j!(j_1+n+l)!(j_2+k-l)!}{2^{j_1+l}(n!)^3j_1!j_2!(k+j)!}
  \]

Moreover, the expression
above is log-convex with respect to $l$, so that we only need to
control the $l=0$ and $l=k-n$ case.

If $l=0$ we are left with
\[
  \cfrac{(2n)!j!(j_1+n)!(k+j_2)!}{2^{j_1}(n!)^3j_1!j_2!(k+j)!}=2^n\binom{2n}{n}\cfrac{\binom{j_1+n}{n}}{2^{j_1+n}}\,\cfrac{\binom{k+j+j_2}{j_2}}{\binom{k+j+j_2}{j}}\leq 4^n\cfrac{\binom{k+j+j_2}{j_2}}{\binom{k+j+j_2}{j}}.
\]
To conclude, $j$ is closer from $\cfrac{k+j+j_2}{2}$ than $j_2$ since
$j\geq j_2$, so that
$\frac{\binom{k+j+j_2}{j_2}}{\binom{k+j+j_2}{j}}\leq 1$,
hence the claim.

If $l=k-n$, one has
\[
  \cfrac{(2n)!j!(j_1+k)!(j_2+n)!}{2^{j_1+k-n}(n!)^3j_1!j_2!(k+j)!}=2^n\binom{2n}{n}\cfrac{\binom{j_1+k}{k}}{2^{j_1+k}}\,\cfrac{\binom{j_2+n}{n}}{\binom{j+k}{k}}\leq 4^n.
\]
We now consider the case $i_1\geq n$ or $i_2\geq n$. We need to
replace $(i_1+j_1+l)!$ with either $(j_1+l)!i_1!$ or
$(j_1+l+n)!(i_1-n)!$. One has
\begin{align*}
  \cfrac{(j_1+l)!i_1!}{i_1!}=(j_1+l)!&\leq \cfrac{(j_1+l+n)!}{n!}\\
  \cfrac{(j_1+l+n)!(i_1-n)!}{i_1!}&\leq
                                   \cfrac{(j_1+l+n)!i_1!}{i_1!n!}=\cfrac{(j_1+l+n)!}{n!}.
\end{align*}
The same inequalities apply with $i_1,j_1$ replaced with $i_2,j_2$.
Hence, in all cases, we are left with
\[
  \cfrac{(2n)!j!(j_1+n+l)!(j_2+k-l)!}{2^{j_1+l}(n!)^3j_1!j_2!(k+j)!},
\]
which we just proved to be smaller than $4^n$.

This yields
\begin{multline*}
      \|(f\sharp g)_k\|_{C^j}\leq\\
    C_J\|f\|_{S^{r,R}_m}\|g\|_{S^{2r,2R}_m}\cfrac{(2r)^j(2R)^k(j+k)!}{(k+j+1)^m}\sum_{n=0}^k\left(\cfrac{4\gamma_r
        r^2}{R}\right)^n\sum_{l=0}^{k-n}\sum_{i_1,i_2=0}^n\sum_{j_1+j_2\leq
      j}\\\cfrac{(k+j+1)^mi_1^di_2^dj_1^dj_2^d}{(j_1+i_1+l+1)^m(j_2+i_2+k-n-l+1)^m(j+2n-i_1-i_2-j_1-j_2+1)^m}.
  \end{multline*}
  We are almost in position to apply Lemma \ref{prop:lem-hard}; since
  \[
    (k+j+n+1)^m\geq (k+j+1)^m,
  \]
  one has
  \begin{multline*}
      \|(f\sharp g)_k\|_{C^j}\leq\\
    C_J\|f\|_{S^{r,R}_m}\|g\|_{S^{2r,2R}_m}\cfrac{(2r)^j(2R)^k(j+k)!}{(k+j+1)^m}\sum_{n=0}^k\left(\cfrac{4\gamma_r
        r^2}{R}\right)^n\sum_{l=0}^{k-n}\sum_{i_1,i_2=0}^n\sum_{j_1+j_2\leq
      j}\\\cfrac{i_1^di_2^dj_1^dj_2^d(k+j+n+1)^m}{(j_1+i_1+l+1)^m(j_2+i_2+k-n-l+1)^m(j+2n-i_1-i_2-j_1-j_2+1)^m}.
  \end{multline*}
  Applying Lemma \ref{prop:lem-hard} yields, for $m$ large enough
  depending on $d$,
  \[
    \|(f\sharp g)_k\|_{C^j}\leq
    C_J\|f\|_{S^{r,R}_m}\|g\|_{S^{2r,2R}_m}\cfrac{(2r)^j(2R)^k(j+k)!}{(k+j+1)^m}\sum_{n=0}^k\left(\cfrac{4\gamma_r
        r^2}{R}\right)^n.
  \]
  As long as $R\geq 4\gamma_r r^2$, which is possible if $R$ is chosen
  large enough since $\gamma_r$ depends only on $r$, one can conclude:
  \[
    \|(f\sharp g)_k\|_{C^j}\leq
    2^mC_J\|f\|_{S^{r,R}_m}\|g\|_{S^{2r,2R}_m}\cfrac{(2r)^j(2R)^k(j+k)!}{(k+j+1)^m}.
  \]

  At this stage, we are almost done with the proof: we obtained that
  the formal series which corresponds, in the $C^{\infty}$ class, to
  the composition $T_N^{cov}(f)T_N^{cov}(g)$, belongs to the same analytic
  symbol class than $g$.

  This concludes the proof.
  \end{proof}

\subsection{Inversion of covariant Toeplitz operators and the Bergman kernel}
\label{sec:invers-covar-toepl}
In this subsection we prove Theorem \ref{thr:Szeg-gen} as well as the
second part of Theorem \ref{thr:Compo}. To do so, we first show in
Proposition \ref{prop:cov-inv}, as a
reciprocal to Proposition \ref{prop:cov-compo}, that
if $f$ and $h$ are analytic symbols of covariant Toeplitz operators
with $f_0$ non-vanishing, then there exists an analytic symbol $g$
such that \[T_N^{cov}(f)T_N^{cov}(g)=T_N^{cov}(h)+O(e^{-cN}).\] We then
prove in Proposition \ref{prop:pre-inv} that, under the same hypotheses, $T_N^{cov}(f)$, whose image is
almost contained in $H^0(M,L^{\otimes N})$ by Remark
\ref{prop:cov-close-H_N}, is invertible on this space up to an
exponentially small error. Thus, one can
conclude that, on $H^0(M,L^{\otimes N})$, there holds
\[T_N^{cov}(g)=T_N^{cov}(h)(T_N^{cov}(f))^{-1}+O(e^{-cN}).\]
This allows us to prove Theorem \ref{thr:Szeg-gen}, since
by setting $h=f$ one recovers that the Bergman kernel can be written
as $T_N^{cov}(f)(T_N^{cov}(f))^{-1}=T_N(a)$. Then, the second part of Theorem
\ref{thr:Compo} follows from Proposition \ref{prop:cov-inv} by
setting $h=a$.

Following the lines of  Proposition \ref{prop:cov-compo}, let us try to
construct inverses for analytic symbols.

\begin{prop}\label{prop:cov-inv}
  Let $U$ denote a small neighbourhood of the diagonal in $M\times M$ and
  let $f,h\in S^{r_0,R_0}_{m_0}(U)$ be analytic symbols, holomorphic
  in the first variable and anti-holomorphic in the second variable,
  for some $r_0,R_0,m_0$. Suppose that the principal symbol $f_0$ of
  $f$ is bounded away from zero on $U$.

  Then there exists $r,R,m$ as well as $g\in S^{r,R}_m(U)$,
  holomorphic in the first variable, anti-holomorphic in the second variable, such that
  \[
    T_N^{cov}(f)T_N^{cov}(g)=T_N^{cov}(h)+O(e^{-cN}).
    \]
  \end{prop}
  \begin{proof}
    Recalling the proof of Proposition \ref{prop:cov-compo}, let us 
    recover $g$ from $f$ and $h=f\sharp g$. By definition of $h_k$,
    one has
\begin{multline}
      \label{eq:rec-hk}
      g_k(x,\overline{z})f_0(x,\overline{z})J(x,x,\overline{z},\overline{z})=h_k(x,\overline{z})\\-\sum_{n=0}^k\cfrac{\Delta^n_v}{n!}\left(\sum_{\substack{l=0\\l+n>0}}^{k-n}f_l(x,\overline{y}(x,v,\overline{z}))g_{k-n-l}(y(x,v,\overline{z}),\overline{z})J(x,v,\overline{z})\right)_{v=0}.
    \end{multline}
As $f_0$ is bounded away from zero, this indeed defines $g_k$ by induction. Let us try to control $g$ in an analytic space.

    We first let $m$ large enough, and $r\geq 2r_02^{m-m_0}$ as well
    as $R\geq 2R_02^{m-m_0}$. Then, by Lemma \ref{prop:move-m-r-R},
    there exist $C_f,C_h,C_J$ independent of $m,r,R$ such that, for
    every $k\geq 0,j\geq 0$,
    \begin{align*}
      \|f_k\|_{C^j(U)}&\leq
                     C_f\cfrac{(r/2)^{j}(R/2)^k(j+k)!}{(j+k+1)^m}\\
      \|h_k\|_{C^j(U)}&\leq
                     C_h\cfrac{r^{j}r^k(j+k)!}{(j+k+1)^m}\\
      \|J\|_{C^j(U\times U)}&\leq C_J\cfrac{(r/2)^jj!}{(j+1)^m}.
    \end{align*}
    Here $J$ denotes again the Jacobian in the change of variables
    corresponding to the Morse lemma for the phase $\Phi_1$.

    We first note that
    \[
      g_0(x,\overline{z})=f_0(x,\overline{z})^{-1}h_0(x,\overline{z})J(x,x,\overline{z},\overline{z}),
    \]
    so that, by Lemma \ref{prop:alg-holo-func}, there exists $C_0$
    such that, for every $r\geq 2r_02^{m-m_0}$ and $R\geq
    2R_02^{m-m_0}$, for  every
    $j\geq 0$,
    \[
      \|g_0\|_{C^j(U)}\leq C_0\cfrac{r^jj!}{(j+1)^m}.
    \]
    
    Let us prove by induction on $l\geq 1$ that, for some fixed $C_g,m,r,R$,
    for every $j\geq 0$, one has
    \[
      \|g_l\|_{C^j}\leq C_g\cfrac{r^jR^l(j+l)!}{(j+l+1)^m}.
    \]
    Over the course of the induction, we will fix the values of $C_g,m,r,R$.

    Suppose that a control above is true for indices up to
    $l=k-1$. Then, from the recursive formula
    \eqref{eq:rec-hk},
    if we repeated the proof of Proposition \ref{prop:cov-compo}, we
    would obtain
    \[
      \|g_k\|_{C^j}\leq C(C_h+C_gC_fC_J)\cfrac{r^jR^k(j+k)!}{(j+k+1)^m}.
    \]
    This is not enough, as the constant $C(C_h+C_gC_fC_J)$ appearing here might be
    greater than $C_g$. However, as we will see, the constant can be
    made arbitrarily small by choosing $C_g$ large enough, as well as $m$ large enough, depending on
    $f$, and $R/r^2$ large enough.

    Let $C_1=C\|(f_0J)^{-1}\|_{H(m,r,U)}$ where $C$ is the constant
    appearing in Proposition \ref{prop:alg-holo-func}.
    
    There holds \[C_h\leq \frac{C_g}{4C_1}\] if $C_g$ is large enough with respect to $C_h,C_f,C_J,C_0$. It remains to estimate the second
    term on the right-hand side of \eqref{eq:rec-hk}.

    Let us isolate the $n=0,l=k$ term in \eqref{eq:rec-hk}. This term is $-g_0Jf_k$, and the $S^{r,R}_m(U)$-norm
    of $g_0Jf$ is
    again smaller than $\frac{C_g}{4C_1}$ if $C_g$ is large enough with
    respect to $C_fC_0C_J$.

    Repeating the proof of Proposition \ref{prop:cov-compo}, the
    $n=0,l<k$ terms in \eqref{eq:rec-hk} are bounded in $C^j$-norm by
    \begin{multline*}
      CC_JC_fC_g\cfrac{r^jR^k(j+k)!}{(j+k+1)^m}\\ \times\sum_{l=1}^{k-1}\sum_{j_1+j_2\leq
        j}\cfrac{(j+k+1)^m}{(j_1+l+1)^m(j_2+k-l+1)^m(j-j_1-j_2+1)^m}.
    \end{multline*}
    
    By Lemma \ref{prop:lem-hard}, since no term in the sum
    \begin{multline*}
      \sum_{\substack{1\leq l\leq k-1\\j_1+j_2\leq
        j}}\cfrac{(j+k+1)^m}{(j_1+l+1)^m(j_2+k-l+1)^m(j-j_1-j_2+1)^m}\\=\sum_{\substack{i_1+i_2+i_3=j+k\\i_1\geq
          1\\i_2\geq 1}}\cfrac{(j+k+1)^m}{(i_1+1)^m(i_2+1)^m(i_3+1)^m}
    \end{multline*}
    contribute as $1$, by Lemma \ref{prop:lem-hard} (with $d=0$ and $n=3$), this sum is smaller than $C(3/4)^{m}$ for some
    $C>0$. Hence, if $m$ is large enough, this contribution is also
    smaller than $\frac{C_g}{4C_1}$. Now $m$ is fixed.

    It remains to control the $n\geq 1$ terms in \eqref{eq:rec-hk}. From the proof of
    Proposition \ref{prop:cov-compo},
    their sum is smaller than
    \[
      CC_JC_fC_g\sum_{n=1}^k\cfrac{r^jR^k(j+k)!}{(j+k+1)^m}\left(\cfrac{4\gamma_r r^2}{R}\right)^n.
    \]
    As long as $R/r^2$ is large enough with respect to $\gamma_r
    C_JC_f$, (which is possible if $R$ is large enough since
    $\gamma_r=Cr$ for some fixed $C$), this is again smaller than
    $\frac{C_g}{4C_1}$.

    In conclusion,
    \[
      \|g_kf_0J\|_{C^j}\leq
      \cfrac{C_g}{C_1}\,\cfrac{r^jR^k(j+k)!}{(j+k+1)^m}.
    \]
    In particular, by Lemma \ref{prop:alg-holo-func}, and since
    $\|(f_0J)^{-1}\|_{H(m,r,U)}=C_1/C$, one has
    \[
      \|g_k\|_{C^j}=\|g_kf_0J(f_0J)^{-1}\|_{C^j}\leq
      C_g\cfrac{r_jR^k(j+k)!}{(j+k+1)^m}.
    \]
    This concludes the induction.

    Once the formal series $g$ is controlled in an analytic symbol
    space, the composition $T_N(g)T_N(f)$ coincides with $T_N(h)$ up
    to an exponentially small error as in the end of the proof of
    Proposition \ref{prop:cov-compo}, hence the claim.
  \end{proof}

  \begin{prop}\label{prop:pre-inv}
    Let $f$ be a function on $U$, holomorphic with respect to the
    first variable, anti-holomorphic with respect to the second variable. If $f$
    is nonvanishing then $S_NT_N^{cov}(f)$ has an inverse on
    $H^0(M,L^{\otimes N})$, with operator norm bounded independently of $N$.
  \end{prop}
  \begin{proof}

    One can invert $S_NT_N^{cov}(f)$ by a formal covariant symbol, that is,
    up to an $O(N^{-K})$ error for any fixed $K$. In particular, there exists
    an operator $A_N$ on $H^0(M,L^{\otimes N})$ such that
    $A_NS_NT_N^{cov}(f)=S_N+O(N^{-1}),$ and such that the operator norm of $A_N$
    is bounded independently on $N$.

    Since $A_NS_NT_N^{cov}(f)$ is invertible on $H^0(M,L^{\otimes N})$,
    so is $S_NT_N^{cov}(f)$, and the operator norm
    of this inverse is $\|A_N\|_{L^2\to L^2}(1+O(N^{-1}))$, which is bounded
    independently on $N$, hence the claim.
  \end{proof}
  
Let us now conclude the proofs of Theorems \ref{thr:Szeg-gen} and \ref{thr:Compo}.
  
    Let $U$ be a small neighbourhood of the diagonal in $M\times
    M$ and let $f$ be any
    function on $U$ bounded away from zero, holomorphic in the first
    variable, anti-holomorphic in the second variable.
    From Proposition \ref{prop:cov-inv} there exists an analytic
    symbol $a$ with the same properties, such
    that \begin{equation}\label{eq:fa=f}
      T_N^{cov}(f)T_N^{cov}(a)=T_N^{cov}(f)+O(e^{-cN}).
    \end{equation}
    
    Let $A_N=(S_NT_N^{cov}(f))^{-1}$ on $H^0(M,L^{\otimes N})$; we
    know from Proposition \ref{prop:pre-inv} that $A_N$ is
    well-defined and bounded independently on $N$. Then,
    for any $u\in
    H^0(M,L^{\otimes N})$, one has
    \[
      T_N^{cov}(a)u=u+O(e^{-cN}).
    \]
    Indeed, one can write $u=A_Nv$ and apply the adjoint of
    (\ref{eq:fa=f}) to $v$.
    
    Moreover, by Remark \ref{prop:cov-close-H_N}, there holds
    \[
      (I-S_N)T_N^{cov}(a)=O(e^{-cN}).
    \]
    To conclude, one has $T_N^{cov}(a)=S_N+O(e^{-cN})$. In other terms,
    \[
      S_N(x,\overline{y})=\Psi^N(x,{y})\sum_{k=0}^{cN}N^{d-k}a_k(x,\overline{y})+O(e^{-cN}).\]
    This concludes the proof of Theorem \ref{thr:Szeg-gen}.

    Let us complete the proof of Theorem \ref{thr:Compo}. Its first
    part is Proposition \ref{prop:cov-compo}. For the second part, we
    apply Proposition \ref{prop:cov-inv} with $h=a$, the symbol of the
    Bergman kernel.

    We remark that, from this proof, $a$ may depend on $f$, but
    necessarily $a$ coincides with the formal symbol of the Bergman
    kernel, so that it is uniquely defined.

  \begin{rem}[Normalised covariant Toeplitz operators]
    Let $T_N^{cov}(a)$ denote the approximate Bergman kernel
    constructed in the previous proposition.
    Once the symbol $a$ is known, one can study, as in the proof of
    Proposition \ref{prop:bound-nb-deriv}, \emph{normalised
      covariant} Toeplitz operators, of the form
    \[
      \Psi^N(x,{y})\left(\sum_{k=0}^{cN}N^{-k}a_k(x,\overline{y})\right)\left(\sum_{k=0}^{cN}N^{d-k}f_k(x,\overline{y})\right).
    \]
    Under this convention, the operator associated with the function
    $f=1$ is $S_N+O(e^{-cN})$, as in contravariant Toeplitz quantization.

    Propositions \ref{prop:cov-compo} and \ref{prop:cov-inv} can be
    adapted to normalised covariant Toeplitz operators, for which the
    algebra product is
    \[
      (f,g)\mapsto ((f* a)\sharp (g* a))* a^{* -1}.
    \]
    For instance, since the Cauchy product is continuous on each symbol class, there
    holds, for $m$ large enough, $r>2^m$ and $R>Cr^3$,
    \[
      \|((f* a)\sharp (g* a))* a^{*
        -1}\|_{S^{2r,2R}_m(U)}\leq
      C_a\|f\|_{S^{r,R}_m(U)}\|g\|_{S^{2r,2R}_m(U)}.
    \]
  \end{rem}

  To conclude this section, we prove that analytic contravariant Toeplitz
  opeartors are contained within analytic covariant Toeplitz
  operators.

    \begin{prop}\label{prop:Contra-to-cov}
  Let $f$ be a real-analytic function on $M$. There exists an analytic
  symbol $g$ and $c>0$ such that
  \[
    T_N(f)=T_N^{cov}(g)+O(e^{-cN}).
    \]
\end{prop}
\begin{proof}
  Recall from Theorem \ref{thr:Szeg-gen} that there exists an analytic
  symbol $a$ such that
  \[
    S_N=T_N^{cov}(a)+O(e^{-cN}).
  \]
  Letting $\widetilde{f}$ be a holomorphic extension of $f$, the kernel of $T_N(f)=S_NfS_N$ is then
  \begin{multline*}
    (x,z)\mapsto \\\Psi^N(x,{z})\int_{y\in
      M}e^{-N\Phi_1(x,y,\overline{y},\overline{z})}\sum_{k,j=0}^{cN}N^{2d-k-j}a_k(x,\overline{y})a_j(y,\overline{z})\widetilde{f}(y,\overline{y})\dd
    y\\+O(e^{-cN}).
  \end{multline*}
  One can then repeat the proof of Proposition \ref{prop:cov-compo}
  with $J$ replaced with $(x,y,\overline{y},\overline{z})\mapsto
  J(x,y,\overline{y},z)\widetilde{f}(y,\overline{y})$. This yields an analytic
  symbol $g$ such that
  \begin{multline*}
    g_k(x,\overline{z})=\sum_{n=0}^k\cfrac{\widetilde{\Delta}^n_{v}}{n!}\left(\sum_{l=0}^{k-n}a_l(x,\overline{y}(x,v,\overline{z})a_{k-n-l}(y(x,v,\overline{z}),\overline{z})\right.\\
    \left.\vphantom{\sum_{n=0}^k}\times J(x,(y,\overline{y})(x,v,\overline{z}),\overline{z})\widetilde{f}((y,\overline{y})(x,v,\overline{z}))\right)_{v=0},
  \end{multline*}
  that is,
  \[
    T_N^{cov}(g)=S_NfS_N+O(e^{-cN}).
  \]
\end{proof}

\subsection{Exponential decay of low-energy states}
\label{sec:expon-decay-low}
Since covariant analytic Toeplitz operators form an algebra up to
exponentially small error terms (Theorem \ref{thr:Compo}), and since
contravariant Toeplitz operators are a
subset of covariant analytic Toeplitz operators (Proposition
\ref{prop:Contra-to-cov}), one can study exponential
localisation for eigenfunctions of contravariant analytic Toeplitz operators. In this subsection we prove Theorem \ref{thr:Gen-exp-decay}.

Let $f$ be a real-analytic, real-valued fnuction on $M$, let $E\in \R$ and let $(u_N)_{N\geq 1}$
be a normalized family of eigenstates of $T_N(h)$ with eigenvalue
$\lambda_N=E+o(1)$. Let $V$ be an open set at positive distance from $\{f=E\}$.
Let $a\in C^{\infty}(M,\R^+)$ be such that $\supp(a)\cap
\{f=E\}=\emptyset$ and $a=1$ on $V$.
The function $a$ is of course not real-analytic; we will nevertheless prove
that
\[
  T_N(a)u_N=O(e^{-cN}).
\]
This implies Theorem \ref{thr:Gen-exp-decay}, since
\[
  \int_V|u_N|^2=\langle u_N,\mathbb{1}_Vu_N\rangle\leq \langle
  u_N,au_N\rangle=\langle u_N,T_N(a)u_N\rangle=O(e^{-cN}).
  \]
Let $W$ be an open set of $M$ such that \[\supp(a)\subset\subset
  W\subset\subset\{f\neq E\}.\]

On $W$, the function $b-E$ is bounded away from zero. Let us
consider, on a neighbourhood of $\diag(W)$ in $M\times M$, the
analytic covariant symbol $g$ which is such that $T_N^{cov}(g)$ is the
analytic inverse (on this neighbourhood) of $T_N(f-\lambda(N))$. This
symbol is well-defined: 
one can check that the construction of
an inverse symbol in Proposition \ref{prop:cov-inv} only relies on local properties. The function
$f-\lambda(N)$ might not be a classical analytic symbol, since we made
no assumption on the eigenvalue $\lambda(N)$. However,
for every $t$ close to $E$ one can define the microlocal inverse $g_t$
of $f-t$ near $W$, in an analytic class independent of $t$, so that we
define the microlocal inverse of $T_N(f-\lambda(N))$ as the 
operator with kernel
\[
  T_N^{cov}(g):(x,y)\mapsto \Psi^N(x,y)g_{\lambda(N)}(N)(x,y).
  \]

We arbitrarily cut
off $g$ outside a neighbourhood of $\diag(W_1)$, where $W\subset\subset
W_1\subset\subset \{f\neq E\}$ so that $T_N^{cov}(g)$
is a well-defined operator. 
Let us prove that, for some $c>0$ small, one has
\[T_N(a)T_N^{cov}(g)T_N(f-\lambda_N)=T_N(a)+O(e^{-cN}).\]
By construction, uniformly on $x\in
W_1$ and $z\in M$, one has
\[
  \int_{y\in M}T_N^{cov}(g)(x,y)T_N(f-\lambda_N)(y,z)=S_N(x,z)+O(e^{-cN}).
\]
In particular, since $T_N(a)$ is bounded by $O(e^{-cN})$ on $W\times
(M\setminus W_1)$, for $x\in W$ one has
\begin{multline*}
  \int_{y_1\in M,y_2\in
    M}T_N(a)(x,y_1)T_N^{cov}(g)(y_1,y_2)T_N(f)(y_2,z)\\
  = \int_{y_1\in W_1,y_2\in
    M}T_N(a)(x,y_1)T_N^{cov}(g)(y_1,y_2)T_N(f-\lambda_N)(y_2,z)+O(e^{-cN})\\
  =\int_{y_1\in W_1}T_N(a)(x,y_1)S_N(y_1,z)+O(e^{-cN})\\
  =\int_{y_1\in M}T_N(a)(x,y_1)S_N(y_1,z)+O(e^{-cN})=T_N(a)(x,z)+O(e^{-cN}).
\end{multline*}
Moreover, uniformly on $(x\notin W,y\in M)$ there holds $T_N(a)(x,y_1)=O(e^{-cN})$ so that, finally,
\[
  T_N(a)T_N^{cov}(g)T_N(f-\lambda_N)=T_N(a)+O(e^{-cN}).
  \]

In particular,
\[
  0=T_N(a)T_N^{cov}(g)T_N(f-\lambda(N))u_N=T_N(a)u_N+O(e^{-cN}),
\]
which concludes the proof.

\section{Acknowledgements}
\label{sec:acknowledgements}

The author thanks L. Charles, N. Anantharaman, S. V\~u Ng{\d o}c and
J. Sj\"ostrand for useful discussion.




\appendix

\section{The Wick rule}
\label{sec:wick-rule}

Here we present a self-contained proof of Proposition
\ref{prop:bound-nb-deriv}.

It is well-known (see \cite{charles_berezin-toeplitz_2003}, Theorem
  2) that there exists an invertible formal series $a$ of functions defined on
  a neighbourhood of the diagonal in $M\times M$, holomorphic in the
  first variable and anti-holomorphic in the second variable, which correspond to the
  Bergman kernel, that is, such that
  \[T_N^{cov}(a)=S_N+O(N^{-\infty}).\]

  In Theorem \ref{thr:Szeg-gen}, we prove that $a$ is in fact
  an analytic symbol; but for the moment, it is sufficient to know that $a$
  exists as a formal series.
  
  Let us deform  covariant Toeplitz operators by this formal symbol $a$, into
  \emph{normalised} covariant Toeplitz operators of the form $T_N^{cov}(f* a)$.
  Here $*$ denotes the Cauchy product of symbols (Proposition \ref{prop:anal-symb-class}). Since in this case $f$ and $g$ are simply
  holomorphic functions one has $f* a=fa$ and $g* a=ga$.
  
  We will first prove our claim for this modified quantization: that
  is, there exists a sequence of bidifferential operators
  $(C_k)_{k\geq 0}$ acting on functions on a neighbourhood
  of the diagonal in $M\times M$, such that, given two
  such functions $f$ and $g$, if we let
  \[
    f\sharp g=\sum_{k=0}^{+\infty}N^{-k}C_k(f,g)+O(N^{-\infty}),
  \]
  then
  \[
    T_N^{cov}((f\sharp g)* a)=T_N^{cov}(fa)T_N^{cov}(ga)+O(N^{-\infty}).
  \]
  Moreover, $C_k$ is of order at most $k$ in each of its arguments.
  Then, we will relate the coefficients $C_k$ with the coefficients
  $B_k$ in the initial claim.

  The claim is easier to prove for the coefficients $C_k$ because
  normalised covariant Toeplitz quantization follows the Wick
  rule. Indeed, if the function $f$, near a point $x_0$, depends only on the first variable (that is, the
  restriction of $f$ to the diagonal is, near this point, a holomorphic
  function on $M$), then the kernel $T_N^{cov}(a f)(x,y)$, for $x$ close
  to $x_0$, can be written as
  $f(x)T_N^{cov}(a)(x,y)=f(x)S_N(x,y)+O(N^{-\infty})$. In particular,
  for $x$ close to $x_0$ the Wick rule holds:
  \[
    T_N^{cov}(af)T_N^{cov}(ag)(x,y)=T_N^{cov}(afg)(x,y)+O(N^{-\infty}),
  \]
  since by Remark \ref{prop:cov-close-H_N} the kernel of $T_N^{cov}(ag)$ is almost holomorphic in the
  first variable, up to an $O(N^{-\infty})$ error. Thus, locally where
  $f$ depends only on the first variable, there holds
  \[
    \forall k\geq 1,\,C_k(f,g)=0.
  \]
  More generally, we wish to compute
  \[
    N^{2d}\Psi^N(x,z)\int_{M}\exp(N\Phi_1(x,y,\overline{y},\overline{z}))(fa)(N)(x,\overline{y})(ga)(N)(y,\overline{z})\dd y,
  \]
  where we recall that
  \[
    \Phi_1(x,y,\overline{w},\overline{z})=-2\widetilde{\phi}(x,\overline{w})+2\widetilde{\phi}(y,\overline{w})-2\widetilde{\phi}(y,\overline{z})+2\widetilde{\phi}(x,\overline{z}).
  \]
  Here, we write $(fa)(N)(x,\overline{y})$ to indicate that $fa$ is
  holomorphic in the first variable and anti-holomorphic in the second
  variable. Similarly, we write
  $\Phi_1(x,y,\overline{w},\overline{z})$ to indicate that $\Phi_1$ is
  a function on $M_x\times \widetilde{M}_{y,\overline{w}}\times M_{z}$, holomorphic in its
  two first arguments and anti-holomorphic in the third argument; we
  integrate over $ M$ which is the subset of
  $\widetilde{M}$ such that $\overline{w}=\overline{y}$.

  First of all, let us prove a Schur test: operator with kernels of
  the form
  \[
    (x,z)\mapsto
    N^{2d}\int_M\exp(N\Phi_1(x,y,\overline{y},\overline{z}))b(x,y,\overline{y},z)\dd
    y
  \]
  are bounded from $L^2(M,L^{\otimes N})$ to itself independently on $N$; in
  particular, successive integration by parts on $(y,\overline{y})$,
  which will introduce
  negative powers of $N$ in the symbol, will lead to a control of the operator.

  Since for any $(x,z)\in U$ one has $|\Psi^N(x,z)|\leq
  e^{-cN\dist(x,z)^2}$, then there exists $C>0$ such that, for any analytic symbol $b$ on $U\times
  U$, there holds
  \begin{align*}
    &N^{2d}\sup_{x}\int_{M}\left|\Psi^N(x,z)\int_M\exp(N\Phi_1(x,y,\overline{y},\overline{z}))b(N)(x,y,\overline{y},z)\dd
    y\right| \dd z
    \\ \leq &
              N^{2d}\sup_{U\times
              U}|b(N)|\sup_x\int_M\int_M|\Psi^N(x,y)||\Psi^N(y,z)|\dd
              y \dd z\\
    \leq & \sup_{U\times U}|b(N)|N^{2d}\sup_x\int_{M\times
           M}e^{-Nc\dist(x,y)^2-Nc\dist(y,z)^2}\dd y \dd z\\
    \leq & C\sup_{U\times U}|b(N)|.
  \end{align*}
  In particular, by the Schur test, the operator with the kernel above
  is bounded independently on $N$.

As $\partial_y \Phi_1$ vanishes in a non-degenerate way at
$\overline{w}=\overline{z}$, one can write
  \[
    f(x,\overline{w})=f(x,\overline{z})-\partial_y\Phi_1\cdot F_1(x,\overline{z},y,\overline{w}).
  \]
  Thus,
  \begin{multline*}
    N^{2d}\Psi^N(x,z)\int_{M}e^{N\Phi_1(x,y,\overline{y},\overline{z})}(fa)(N)(x,\overline{y})(ga)(N)(y,\overline{z})\dd y\\=
    N^{2d}\Psi^N(x,z)f(x,\overline{z})
    \left(\int_Me^{N\Phi_1(x,y,\overline{y},\overline{z})}a(N)(x,\overline{y})(ga)(N)(y,\overline{z})\dd
      y\right.\\ \left.+N^{-1}
    \int_Me^{N\Phi_1(x,y,\overline{y},\overline{z})}a(N)(x,\overline{y})\partial_M\left[F_1(x,\overline{z},y,\overline{y})(g
      a)(N)(y,\overline{z})\right]\dd y\right).
  \end{multline*}
  The first term in the right-hand side above is equal to
  \[
    f(x,\overline{z})\int_M T_N^{cov}(a)(x,\overline{y})T_N^{cov}(g
    a)(y,\overline{z})\dd y=f(x,\overline{z})T_N^{cov}(g
    a)(x,\overline{z})+O(N^{-\infty}),
  \]
  since $T_N^{cov}(a)=S_N+O(N^{-\infty})$.
  
  In the second line, which is of order $N^{-1}$ by a Schur test, derivatives of $g$
  of order at most $1$ appear.
  This remainder can be written as
  \begin{multline*}
    N^{2d-1}\Psi^N(x,{z})\int_Me^{N\Phi_1(x,y,\overline{y},\overline{z})}a(N)(x,\overline{y})\left[\partial_yF_1(x,\overline{z},y,\overline{y})\right](g
      a)(N)(y,\overline{z})\dd y\\
    +
    N^{2d-1}\Psi^N(x,{z})\int_Me^{N\Phi_1(x,y,\overline{y},\overline{z})}a(N)(x,\overline{y})F_1(x,\overline{z},y,\overline{y})[\partial_y(ga)(N)(y,\overline{z})\dd y.
  \end{multline*}
  We recover the initial expression, where $f$ has been replaced with
  either $F_1$ or $\partial_yF_1$, and $g$ has potentially been
  differentiated once. Thus, by induction, the coefficient $C_k(f,g)$ only differentiates at most
  $k$ times on $g$. By duality, $C_k(f,g)$ only
  differentiates at most $k$ times on $f$.

  Let us now relate the coefficients $C_k$ and $B_k$. Let $a^{*
    -1}$ denote the inverse of $a$ for the Cauchy product. One has
  \begin{align*}
    T_N^{cov}(f)T_N^{cov}(g)&=T_N^{cov}((f a^{* -1})*
    a)T_N^{cov}((g a^{* -1})*
    a)+O(N^{-\infty})\\ &=T_N^{cov}((C_k(f,g))_{k\geq 0}*
    a)+O(N^{-\infty}),
  \end{align*}
  so that the
  coefficients $B_k$ in the initial claim are recovered as
  \[
    B_k(f,g)=\sum_{j+l+m\leq k}a_jC_{k-j-l-m}(fa^{* -1}_l,ga^{* -1}_m),
  \]
  thus $B_k$ itself differentiates at most $k$ times on $f$ and at most $k$
  times on $g$.

\bibliographystyle{abbrv}
\bibliography{main}
\end{document}